\documentclass[11pt,reqno]{amsart}

\usepackage{enumerate,url,amssymb,  mathrsfs}%, pdfsync}% ,showkeys, pdfsync}
\usepackage{graphicx}
\usepackage{esint,comment}
\pdfoutput=1

\newtheorem{theorem}{Theorem}[section]
\newtheorem{lemma}[theorem]{Lemma}
\newtheorem*{lemma*}{Lemma}
\newtheorem{proposition}[theorem]{Proposition}
\newtheorem{corollary}[theorem]{Corollary}

\theoremstyle{definition}

\theoremstyle{remark}
\newtheorem{remark}[theorem]{Remark}

\numberwithin{equation}{section}

\DeclareMathOperator{\LIL}{LIL}

\newcommand{\C}{\mathbb{C}}

\newcommand{\DD}{\mathbb{D}}

\newcommand{\N}{\mathbb{N}}

\DeclareMathOperator{\re}{Re}
\DeclareMathOperator{\var}{Var}
\DeclareMathOperator{\Aut}{Aut}
\DeclareMathOperator{\im}{Im}
\DeclareMathOperator{\hyp}{hyp}

\DeclareMathOperator{\Lip}{Lip}

\DeclareMathOperator{\Hdim}{H.dim }
\DeclareMathOperator{\Mdim}{M.dim }

\DeclareMathOperator{\supp}{supp }
\DeclareMathOperator{\Area}{Area }
\DeclareMathOperator{\ei}{I}
\DeclareMathOperator{\pp}{PP}
\DeclareMathOperator{\WP}{WP}
\DeclareMathOperator{\bl}{B}
\DeclareMathOperator{\fuchs}{F}
\DeclareMathOperator{\shell}{shell}

\DeclareMathOperator{\Vpp}{\mathcal V_{\pp}}

\def\XXint#1#2#3{{\setbox0=\hbox{$#1{#2#3}{\int}$}
\vcenter{\hbox{$#2#3$}}\kern-.5\wd0}}

\def\le{\leqslant}
\def\ge{\geqslant}
\begin{document}
\baselineskip6mm
\vskip0.4cm
\title[Asymptotic variance of the Beurling transform]{Asymptotic variance\\ of the Beurling transform}

\author[K. Astala]{Kari Astala}
\address{Le Studium$^\circledR$, Loire Valley Institute for Advanced Studies, Orl\'eans \& Tours, France; Mapmo, rue de Chartres, 45100 Orl\'eans, France; 
Department of Mathematics and Statistics, University of Helsinki, 
         P.O. Box 68, FIN-00014, Helsinki, Finland}
\email{kari.astala@helsinki.fi}

\author[O. Ivrii]{Oleg Ivrii}
\address{Department of Mathematics and Statistics, University of Helsinki, 
         P.O. Box 68, FIN-00014, Helsinki, Finland}
\email{oleg.ivrii@helsinki.fi}

\author[A. Per\"al\"a]{Antti Per\"al\"a}
\address{Department of Mathematics and Statistics, University of Helsinki, 
         P.O. Box 68, FIN-00014, Helsinki, Finland}
\email{antti.i.perala@helsinki.fi}

\author[I. Prause]{Istv\'an Prause}
\address{Department of Mathematics and Statistics, University of Helsinki,
         P.O. Box 68, FIN-00014, Helsinki, Finland}
\email{istvan.prause@helsinki.fi}

\thanks{
A.P. was supported by the Vilho, Yrj\"o and Kalle V\"ais\"al\"a Foundation and the Emil Aaltonen Foundation. I.P. and A.P. were supported by the Academy of Finland (SA) grants 1266182 and 1273458. All authors were supported by the Center of Excellence Analysis and Dynamics, SA grants 75166001 and 12719831. Part of this research was performed while K.A. and I.P. were visiting the Institute for Pure and Applied Mathematics (IPAM), which is supported by the National Science Foundation.}


%\date{\today}


\begin{abstract} 
We study the interplay between infinitesimal deformations of conformal mappings, quasiconformal distortion estimates and integral means spectra.
By the work of McMullen, the second derivative of  the Hausdorff dimension of the boundary of the image domain
is naturally related to asymptotic variance of the Beurling transform. In view of a theorem of Smirnov which states that the dimension of a $k$-quasicircle is at most $1+k^2$, it is natural to expect that the maximum asymptotic variance $\Sigma^2 = 1$. 
In this paper, we prove $0.87913 \le \Sigma^2 \le 1$.

For the lower bound, we give examples of polynomial Julia sets which are $k$-quasicircles with dimensions $1+ 0.87913 \, k^2$ for $k$ small, thereby showing that
$\Sigma^2 \ge 0.87913$.
The key ingredient in this construction is a good estimate for the distortion $k$, which is better than the one given by a straightforward use of the $\lambda$-lemma in the appropriate parameter space.
Finally, we develop a new fractal approximation scheme for evaluating $\Sigma^2$ in terms of nearly circular polynomial Julia sets.
\end{abstract}

\maketitle

% \tableofcontents

\section{Introduction}\label{se:introduction}

In his work on the Weil-Petersson metric \cite{mcmullen}, McMullen considered certain holomorphic families of conformal maps 
$$
\varphi_t \colon \mathbb{D^*} \to \mathbb{C}, 
\qquad\varphi_0(z) = z, \qquad \text{where }\mathbb{D^*}=\{z : |z|>1\},
$$
 that  naturally arise in complex dynamics and Teichm\"uller theory. 
For these special families, he used thermodynamic formalism to relate a number of different dynamical features. For instance, he showed that the infinitesimal growth of the Hausdorff dimension of the Jordan curves $\varphi_t(\mathbb{S}^1)$ is connected to the asymptotic variance of the first derivative of the vector field $v=\frac{d\varphi_t}{dt}\bigl |_{t=0}$  by  the formula
\begin{equation} \label{eq:mcmullen}
2\, \frac{d^2}{dt^2}\biggl |_{t=0} \Hdim \, \varphi_t(\mathbb{S}^1) = \sigma^2(v'), \quad
\end{equation}
where the \emph{asymptotic variance} of a Bloch function $g$ in $\mathbb{D^*} $ is given by %-- see Section 2 -- as
\begin{equation}
\label{eq:asymvariance} 
\sigma^2(g)=\frac{1}{2\pi}\limsup_{R\to1^+}\,  \frac{1}{|\log(R-1)|} \int_{|z|=R} |g(z)|^2 |dz|. 
\end{equation}
This terminology is justified by viewing $g$ as a stochastic process 
$$Y_s(\zeta)=g((1-e^{-s}) \zeta), \quad \zeta \in \mathbb{S}^1, \quad 0\leqslant s <\infty,
$$
with respect to the probability measure $|d\zeta|/2\pi$, in which case $\sigma^2(g)=\limsup_{s\to \infty} \frac{1}{s} \, \sigma^2_{Y_s}$.
For further relevance of probability methods to the study of the  boundary distortion of conformal maps,  we refer the reader to \cite{kayumov, makarov90}.

For arbitrary families of conformal maps, the identity \eqref{eq:mcmullen} may not hold. For instance, Le and Zinsmeister \cite{le-zinsmeister} constructed a family $\{\varphi_t\}$ for which $\sigma^2(v')$ is zero, while $t \mapsto \Mdim \varphi_t (\mathbb{S}^1)$ (with Hausdorff dimension replaced by Minkowski dimension) is equal to 1 for $t < 0$ but grows quadratically for $t > 0$. 

Nevertheless, it is natural to enquire if McMullen's formula \eqref{eq:mcmullen}  holds on the level of universal bounds. As will be explained in detail in the subsequent sections, for general  holomorphic families of conformal maps $\varphi_t $ parametrised  by a complex parameter $t \in \DD$, one can combine the work of  Smirnov  \cite{smirnov}  with the theory of holomorphic motions \cite{MSS, Slod} to show that  
\begin{equation}
\label{eq:smirnov-intro} 
\Hdim\, \varphi_t(\mathbb{S}^1) \leqslant 1 + \frac{(1-\sqrt{1-|t|^2})^2}{|t|^2}  
= 1+ \frac{|t|^2}{4}+ \mathcal{O}(|t|^4) , \quad t \in \DD.
\end{equation}
It is conjectured that the  equality in \eqref{eq:smirnov-intro} holds  for some family, but this is still  open.
On the other hand, the derivative of the infinitesimal vector field $v=\frac{d\varphi_t}{dt}\bigl |_{t=0}$ can be represented in the form 
$$ v' = \mathcal{S} \mu
$$
where $|\mu(z)| \leqslant \chi_{\DD}$ and $\mathcal{S}$ is the {\it Beurling transform}\,, the 
principal value integral
\begin{equation}
\label{eq:beurling}
\mathcal{S}\mu(z)=-\frac{1}{\pi} \int_{\mathbb{C}}\frac{\mu(w)}{(z-w)^2}dm(w).
\end{equation}

\smallskip

\noindent (Since the support of $\mu$ is contained in the unit disk,  $v'$ is a holomorphic function on the exterior unit disk.)

In this formalism, McMullen's identity describes the asymptotic variance
$ \sigma^2(\mathcal{S}\mu)  
$
for a ``dynamical'' Beltrami coefficient $\mu$, which is invariant under a co-compact Fuchsian group or a Blaschke product.

In this paper, we study the quantity
\begin{equation} \label{eq:Sigma} 
\Sigma^2 := \;  \sup \{\sigma^2(\mathcal{S}\mu) : |\mu| \leqslant \chi_{\mathbb{D}} \}
\end{equation}
from several different perspectives.
In addition to the problem of dimension distortion of quasicircles,  $\Sigma^2$
is naturally related to questions on integral means of conformal maps, which we discuss later in the introduction.
The first result in this work is an upper bound for $\Sigma^2$:

\begin{theorem} 
\label{thm:upperbound-intro}
Suppose $\mu$ is measurable in $\C$ with $|\mu| \leqslant \chi_{\mathbb{D}}$. Then,
\begin{equation} \label{goal3} 
 \sigma^2(\mathcal{S}\mu) := \frac{1}{2\pi}\limsup_{R\to1^+}\,  \frac{1}{|\log(R-1)|} \int_{0}^{2 \pi} |\mathcal{S} \mu (Re^{i \theta})|^2 \; d\theta \;  \leqslant \;   1.
\end{equation}
\end{theorem}

We give two different proofs for \eqref{goal3}, one using holomorphic motions and quasiconformal geometry in Section \ref{sec:upper}, and another based on complex dynamics and fractal approximation in Section \ref{sec:fractal}.

In view of McMullen's identity and the possible sharpness of Smirnov's dimension bounds, it is natural to expect that the bound \eqref{goal3} is optimal with $\Sigma^2=1$, and in the first version of this paper we formulated a conjecture to that extent. However, after having read our manuscript, H\aa kan Hedenmalm managed to show \cite{hedenmalm} that actually $\Sigma^2 <1$.

For lower bounds on $\Sigma^2$, we produce examples in Section \ref{se:lowerbound} showing:

\begin{theorem} \label{goal5} There exists a Beltrami coefficient $|\mu| \leqslant \chi_{\mathbb{D}}$ such that 
\begin{equation*}
 \sigma^2(\mathcal{S}\mu) >  0.87913.
\end{equation*}
\end{theorem}

In fact, our construction gives new bounds for the quasiconformal distortion of certain polynomial Julia sets:

\begin{theorem} \label{julia} 
Consider the polynomials $P_t(z)=z^d+t\, z$. For $\; |t| < 1$, the Julia set $\mathcal J(P_t)$ is a Jordan curve which can be expressed as the image of the unit circle
by a  $k$-quasiconformal map of $\C$, where
$$
k = \frac{d^{\frac{1}{d-1}}}{4}|t| + \mathcal{O}(|t|^2).
$$
In particular, when $d=20$ and $|t|$ is small, $k \approx 0.585 \cdot \frac{|t|}{2}$ and 
$\mathcal J(P_t)$ is a $k$-quasicircle with 
\begin{equation}
\label{dimbound}
\Hdim \, \mathcal J(P_t) \approx 1+0.87913 \cdot k^2.
\end{equation}
\end{theorem}
 Note that the distortion estimates in Theorem \ref{julia} are strictly better (for $d \ge 3$) than those given by a straightforward use of the $\lambda$-lemma.
 For a detailed discussion, see  Section \ref{se:lowerbound}.
 In terms of the dimension distortion of quasicircles, Theorem \ref{julia} improves upon all previously known examples. For instance,
 the holomorphic snowflake construction of \cite{astalarohdeschramm95} gives a $k$-quasicircle of dimension $\approx 1+0.69\, k^2$.
 
In order to further explicate the relationship between asymptotic variance and dimension asymptotics, consider the function
\begin{equation*}
D(k)= \sup\{ \Hdim \, \Gamma : \Gamma \mbox{ is a $k$-quasicircle} \}, \quad 0\le k <1.
\end{equation*}
The fractal approximation principle of Section \ref{sec:fractal} roughly says that infinitesimally, it is sufficient to consider certain quasicircles, namely 
 nearly circular polynomial Julia sets. 
  As a consequence, we prove: 
\begin{theorem}
\label{thm:sigmaanddimension}
\begin{equation}
 \Sigma^2 \le \liminf_{k \to 0} \frac{D(k)-1}{k^2}.
 \end{equation}
\end{theorem}
Together with Smirnov's bound \cite{smirnov},
\begin{equation}
\label{eq:smirnovdk}
D(k) \le 1+k^2,
\end{equation}
Theorem \ref{thm:sigmaanddimension} gives an alternative proof for Theorem \ref{thm:upperbound-intro}.
We note that the function $D(k)$ may be also characterised in terms of several other properties 
in place of Hausdorff dimension, see \cite{Arevista}. It would be interesting to know if the reverse inequality in Theorem \ref{thm:sigmaanddimension} holds.

In Section \ref{se:Fuchsian}, we study the fractal approximation question in the Fuchsian setting. One may expect that
it may be possible to approximate $\Sigma^2$  using Beltrami coefficients invariant under co-compact Fuchsian groups.
However, this turns out not to be the case. To this end, we show:
\begin{theorem}
\label{fuchsian-case}
\[\Sigma^2_{\fuchs} \, := \,  \sup_{\mu \in M_{\fuchs},\ |\mu| \le \chi_{\mathbb{D}}} \sigma^2(\mathcal S\mu) \, < \, 2/3.\]
\end{theorem}
Theorem \ref{fuchsian-case} may be viewed as an upper bound for the quotient of the Weil-Petersson and Teichm\"uller metrics,
over all Teichm\"uller spaces $\mathcal T_g$ with $g \ge 2$.
(To make the bound genus-independent, one needs to normalise the hyperbolic area of Riemann surfaces to be 1.) The proof follows
from simple duality arguments and the fact that there is a definite defect in the Cauchy-Schwarz inequality.

Finally, we compare our problem with another method of embedding a conformal map $f$ into a flow:
 \begin{equation} \label{flow3}
 \log f_t'(z) = t \log f'(z), \qquad t \in \DD.
\end{equation}
In this case, the derivative of the infinitesimal vector field at $t=0$  is just the Bloch function $\log f'(z)$.
However, even if $f$ itself is univalent, the univalence of $f_t$  is only guaranteed for $|t| \le 1/4$, see \cite{pfaltz}.
One advantage of the notion \eqref{eq:Sigma} and holomorphic flows parametrised by Beltrami equations is that they do not suffer from this ``univalency gap''. 

In the case of domains bounded by regular fractals and the corresponding equivariant Riemann mappings $f(z)$, 
we have  several interrelated dynamical and  geometric characteristics:

\begin{itemize}
\item The {\em integral means spectrum} of a conformal map:
\begin{equation}
\beta_f(\tau) = \limsup_{r \to 1} \frac{\log \int_{|z|=r} |(f')^\tau| d\theta}{\log \frac{1}{1-r}}, \qquad \tau \in \mathbb{C}.
\end{equation}
\item
The {\em asymptotic variance}  a Bloch function $g \in \mathcal B$:
\begin{equation}
\sigma^2(g) = \limsup_{r\to1} \frac{1}{2\pi |\log(1-r)|} \int_{|z|=r} |g(z)|^2 d\theta.
\end{equation}
\item The {\em LIL constant} of a conformal map is defined as the essential supremum of $C_{\LIL}(f, \theta)$ over $\theta \in [0, 2\pi)$ where
\begin{equation}\label{lilil}
C_{\LIL}(f, \theta) = \limsup_{r \to 1} \frac{\log |f'(re^{i\theta})|}{\sqrt{\log \frac{1}{1-r} \log \log\log  \frac{1}{1-r}}}.
\end{equation}
\end{itemize}

\begin{theorem} 
\label{dynamical-connections}
Suppose $f(z)$ is a conformal map, such that the image of the unit circle $f(\mathbb{S}^1)$ is a Jordan curve, invariant under a hyperbolic conformal dynamical system.
Then, 
\begin{equation}\label{1234}
2 \frac{d^2}{d\tau^2}\biggl |_{\tau=0} \beta_f(\tau) = \sigma^2(\log f') = C^2_{\LIL}(f),
\end{equation}
where $\beta(\tau)$ is the integral means spectrum, $\sigma^2$ is the asymptotic variance of the Bloch function $\log f'$, and $C_{\LIL}$ denotes
the constant in the law of the iterated logarithm \eqref{lilil}.
\end{theorem}

We emphasise that the above quantities are not equal in general, but only for special domains $\Omega$ that have fractal boundary.
For these domains, the limits in the definitions of $\beta_f(\tau)$ and $\sigma^2(\log f')$ exist, while $ C_{\LIL}(f, \theta)$ is a constant function
(up to a set of measure 0).

The equalities in \eqref{1234} are mediated by a fourth quantity involving the {\em dynamical asymptotic variance} 
of a H\"older continuous potential 
from thermodynamic formalism. The equality between the dynamical variance and $C^2_{\LIL}$ is established in \cite{PUZ1, PUZ2}, while the works \cite{binder, makarov99} give the connection to the integral means $\beta(\tau)$.
The missing link, it seems, is the connection between the dynamical variance and  $\sigma^2$, which can be proved using a global analogue of
McMullen's coboundary relation. Details will be given in Section \ref{sec:dynamics}. 
We note that an alternative approach connecting $\beta(\tau)$ and $\sigma^2$ directly has been considered in the special case of
polynomial Julia sets, see \cite{kayumov}. 

With these connections in mind, we relate our quantity $\Sigma^2$ to the universal integral means spectrum $B(\tau)= \sup_{f} \beta_f(\tau)$:
\begin{theorem} \label{thm:beta}
\[ \liminf_{\tau \to 0} \frac{B(\tau)}{\tau^2/4} \geqslant \Sigma^2.\]
\end{theorem}
\noindent 
In view of the lower bound for $\Sigma^2$ given by Theorem \ref{goal5},  this improves upon the previous best known lower bound \cite{HK} for the behaviour of the universal integral means spectrum near the origin.
The proof of   Theorem \ref{thm:beta} along with additional numerical advances is presented in Section \ref{sec:dynamics}.

While the two approaches above for constructing flows of conformal maps are somewhat different, there is a relation: singular quasicircles lead to singular conformal maps via welding-type procedures \cite{prause-smirnov}.   The parallels are summarised in Table \ref{table:singularconformalmaps} below, where exact equalities hold only in the dynamical setting.

\smallskip
\begin{table}[htbp]
\renewcommand{\arraystretch}{1.2}
\begin{center}
\begin{tabular}{|c|c|c|}
\hline
Holomorphic motion & $\overline{\partial} \varphi_t =t \, \mu \, \partial \varphi_t$ & 
$\log f_t' = t \log f'$ \\ \hline
Bloch function $v'$  & $\mathcal{S}\mu$ & $\log f'$ \\  \hline
Univalence & $\|\mu\|_\infty \leqslant 1$ & $f$ conformal \\ \hline
$\sigma^2(v') = c$& $\Hdim \varphi_t(\mathbb{S}^1) = 1+{c} \, |t|^2/4 + \ldots$ & $\beta_f(\tau) = {c} \, {\tau^2}/{4} + \ldots$  \\ \hline
Examples & \multicolumn{2}{c|}{Lacunary series}\\ \hline
\end{tabular}
\end{center}
\bigskip
\caption{Singular conformal maps and the growth of Bloch functions}
\label{table:singularconformalmaps}
\end{table}

\section{Bergman projection and Bloch functions}
\label{section:background}

In this section, we introduce the notion of asymptotic variance for Bloch functions and discuss some of its basic properties. 

\subsection{Asymptotic variance.}
The Bloch space $\mathcal B$ consists of analytic functions $g$ in $\mathbb{D}$ which satisfy
\begin{equation*}
 \|g\|_{\mathcal B}:=\sup_{z \in \mathbb{D}} \, (1-|z|^2) |g'(z)| < \infty.
\end{equation*}
Note that $\| \cdot \|_{\mathcal B}$ is only a seminorm on $\mathcal B$. A function $g_0 \in \mathcal B$ belongs to the little Bloch space $\mathcal B_0$ if 
$$\label{lbloch}
\lim_{|z|\to 1^-}(1-|z|^2)|g_0'(z)|=0.
$$ To measure the boundary growth of a Bloch function $g \in \mathcal B$, we define its asymptotic variance by 
\begin{equation}\label{eq:asymvariance2}
 {\sigma}^2(g) :=\frac{1}{2\pi}\limsup_{r\to1^-} \frac{1}{|\log(1-r)|} \int_{0}^{2\pi} |g(re^{i \theta})|^2 d\theta. 
\end{equation} 
Lacunary series provide examples with non-trivial (i.e.~positive) asymptotic variance. For instance,
for $g(z)=\sum_{n=1}^\infty z^{d^n}$ with $d \geq 2$, a quick calculation based on orthogonality shows that
\begin{equation} \label{eq:sigma-lacunary} 
{\sigma}^2(g)= \frac{1}{\log d}.
\end{equation}
Following \cite[Theorem 8.9]{Pomm}, to estimate the asymptotic variance, we use Hardy's identity which says that
\begin{equation} \label{Hardy3}
\left(\frac{1}{4r}\frac{d}{dr}\right)\left(r\frac{d}{dr}\right)\frac{1}{2\pi}\int_0^{2\pi}|g(re^{i\theta})|^2d\theta=\frac{1}{2\pi}\int_0^{2\pi}|g'(re^{i\theta})|^2d\theta 
\end{equation}
\begin{equation} \label{Hardy4}
\leq \|g \|^2_{\mathcal B} \left(\frac{1}{1-r^2}\right)^2 
=\|g \|^2_{\mathcal B} \left(\frac{1}{4r}\frac{d}{dr}\right)\left(r\frac{d}{dr}\right) \log\frac{1}{1-r^2}. \nonumber
\end{equation}
From (\ref{Hardy3}), it follows that  ${\sigma}^2(g) \leqslant  \|g\|_{\mathcal B}^2$. In particular, the asymptotic variance of a Bloch function is finite.
It is also easy to see that adding an element from the  little Bloch space does not affect the asymptotic variance, i.e.~$\sigma^2(g+g_0) = \sigma^2(g)$. 

\subsection{Beurling transform and the Bergman projection}
\label{subsection:SvsP}
For a measurable function $\mu$ with $|\mu| \leqslant \chi_{\mathbb{D}}$, the Beurling transform $g=\mathcal{S}\mu$ is  an analytic function in the exterior disk $\mathbb{D}^*=\{z: |z|>1\}$ which satisfies a Bloch bound of the form $\| g \|_{{\mathcal B}^*} := |g'(z)| (|z|^2-1) \leqslant C$. Note that we use the notation ${\mathcal B}^*$ for  functions in $\mathbb{D}^*$ --
we reserve 
 the symbol $\mathcal B$ for the standard Bloch space in the unit disk $\mathbb{D}$.
By passing to the unit disk, we are naturally led to the Bergman projection 
\begin{equation}
\label{eq:bergman-def}
 P\mu(z)=\frac{1}{\pi} \int_{\mathbb{D}}\frac{\mu(w)dm(w)}{(1-z\overline{w})^2}
\end{equation}
and its action on $L^\infty$-functions.
Indeed, comparing \eqref{eq:beurling} and \eqref{eq:bergman-def}, we see that
$P\mu(1/z) = - \, z^2 \mathcal{S}\mu_0(z)$
for $ \mu_0(w) = \mu(\overline{w})$ and $z \in \mathbb D^*$.
From this connection between the Beurling transform and the Bergman projection, it follows that
\begin{equation} \label{eq:bergman-sigma}
\Sigma^2 = \sup_{|\mu| \le \chi_{\mathbb{D}}} {\sigma}^2(\mathcal S \mu) =  \sup_{|\mu| \le \chi_{\mathbb{D}}} {\sigma}^2(P\mu). 
\end{equation}
In view of the above equation, the Beurling transform and the Bergman projection are mostly interchangeable.
Due to natural connections with the quasiconformal literature, we mostly work with the Beurling transform. However, in this section on a priori bounds, it is preferable to work with the Bergman projection to keep the discussion in the disk.

\subsection{Pointwise estimates}

According to \cite{perala}, the seminorm of the Bergman projection from $L^\infty(\mathbb{D}) \to \mathcal B$ is $8/\pi$. Integrating (\ref{Hardy3}), we get
\begin{equation*}
\frac{1}{2\pi}\int_0^{2\pi}|P\mu(re^{i\theta})|^2d\theta \leq \left(\frac{8}{\pi}\right)^2  \log\frac{1}{1-r^2}, \qquad r \to 1^-,
\end{equation*}
which implies that $\Sigma^2 \le (8/\pi)^2$. 
One can also equip the Bloch space with seminorms that use higher order derivatives
\begin{equation} \label{mnorm}
\|f\|_{\mathcal B,m}=\sup_{z \in \mathbb{D}}(1-|z|^2)^m |f^{(m)}(z)|,
\end{equation}
where $m \geqslant 1$ is an integer.
Very recently, Kalaj and Vujadinovi\'c \cite{kalvuj} calculated the seminorm of the Bergman projection when the Bloch space is equipped with \eqref{mnorm}. According to their result,
\begin{equation} \label{mnorm3}
\|P\|_{\mathcal B,m}=\frac{\Gamma(2+m)\Gamma(m)}{\Gamma^2(m/2+1)}.
\end{equation}
It is possible to apply the differential operator in \eqref{Hardy3} $m$ times and use the pointwise estimates \eqref{mnorm3}. 
In this way, one ends up with the upper bounds 
\begin{equation}
\label{eq:improved-pw-bounds}
 \sigma^2(\mathcal{S}\mu)  = {\sigma}^2(P\mu)  \le \frac{\Gamma(2+m)^2\Gamma(m)^2}{\Gamma(2m)\Gamma^4(m/2+1)}.
\end{equation}
Putting $m=2$ in (\ref{eq:improved-pw-bounds}), one obtains that $\sigma^2(\mathcal{S}\mu)\leq 6$, which is a slight improvement to $(8/\pi)^2$ and is the best upper bound that can be achieved with this argument.  Using quasiconformal methods in Section \ref{sec:upper}, we will show the significantly better upper bound $\sigma^2(\mathcal{S}\mu)\le 1.$

\subsection{C\'esaro integral averages}

In Section \ref{sec:fractal} on fractal approximation, we will need the C\'esaro integral averages from
 \cite[Section 6]{mcmullen}. Following McMullen, for $f \in \mathcal B$, $m \ge 1$ and $r \in [0,1)$,  we define 
%$$\sigma^2_{2m}(f,r)=\frac{1}{\Gamma(2m)|\log(1-r)|}\int_0^r \frac{(1-s^2)^{2m}}{1-s}\left[\frac{1}{2\pi}\int_0^{2\pi}|f^{(m)}(se^{i\theta})|^2d\theta \right] ds$$
\begin{equation*}
\sigma^2_{2m}(f,r)=\frac{1}{\Gamma(2m)} \frac{1}{|\log(1-r)|}\int_0^r \frac{ds}{1-s}\left[\frac{1}{2\pi}\int_0^{2\pi} 
\biggl |(1-s^2)^m f^{(m)}(se^{i\theta}) \biggr |^2 d\theta \right]
\end{equation*}
and
\begin{equation}
\label{eq:mcm-cesaro}
\sigma^2_{2m}(f)=\limsup_{r\to 1^-}\sigma^2_{2m}(f,r).
\end{equation}

We will need \cite[Theorem 6.3]{mcmullen} in a slightly more general form, where we allow the use of ``limsup'' instead of requiring the existence of a limit:

\begin{lemma} \label{lemma:sigmas}
For $f \in \mathcal B$,
\begin{equation}\label{chain}
\sigma^2(f) = \sigma^2_{2}(f) = \sigma^2_4(f) = \sigma^2_6(f) = \ldots
\end{equation}
Furthermore, if the limit as $r \to 1$ in $\sigma^2_{2m}(f)$ exists for some $m \ge 0$, then the limit as $r \to 1$ exists in $\sigma^2_{2m}(f)$ for all $m \ge 0$.
\end{lemma}

The original proof from \cite{mcmullen} applies in this setting.

\section{Holomorphic families} \label{holom}

Our aim is to understand holomorphic families of conformal maps, and the infinitesimal change of Hausdorff dimension.
The natural setup for this is provided by {\it holomorphic motions} \cite{MSS}, maps
$\Phi:\DD\times A\rightarrow \C$, with $A \subset \C$, such that
\begin{itemize}
\item For a fixed $a\in A$,  the map $\lambda\rightarrow \Phi(\lambda,a)$ is
holomorphic in $\DD$.
\item For a fixed $\lambda\in\DD$, the map 
$a\rightarrow \Phi(\lambda,a)
= \Phi_{\lambda}(a) $
is injective.
\item The mapping $\Phi_{0}$ is the identity on $A$,
\[ \Phi(0,a)=a, \quad  \mbox{ for every $a\in A$.} \]
\end{itemize}

It follows from the works of Ma{\~n}{\'e}-Sad-Sullivan \cite{MSS} and Slodkowski \cite{Slod} that each $ \Phi_{\lambda}$  can be extended to a quasiconformal homeomorphism of $\C$. In other words, each $f =  \Phi_{\lambda}$ is a homeomorphic $W^{1,2}_{loc}(\C)$-solution to the {\it Beltrami equation}
$$\overline \partial f(z) = \mu(z) \partial f(z) \quad \mbox{for a.e.~} z \in \C.
$$
Here the {\it dilatation} $\mu(z) = \mu_\lambda(z)$ is measurable in $z \in \C$, and the mapping  $f$ is called {\it $k$-quasiconformal} if $\| \mu \|_{\infty} \leq k < 1$. As a function of  $\lambda \in \DD$, the dilatation $\mu_\lambda$ is  a holomorphic $L^\infty$-valued function with 
$\|Ê\mu_\lambda\|_{\infty} \leq |\lambda|$, see \cite{bersroyden}. In other words, $\Phi_\lambda$ is a $|\lambda|$-quasiconformal mapping.

Conversely, as is well-known, homeomorphic solutions to the  Beltrami equation can be embedded into holomorphic motions. For this work, we shall need a specific and perhaps non-standard representation of the mappings which quickly implies the embedding. For details, see  Section \ref{sec:upper}.

\subsection{Quasicircles.} \label{subsec:quasicircles}
Let us now consider a holomorphic family of conformal maps $\varphi_t \colon \mathbb{D^*} \to \mathbb{C}$, $t \in \mathbb{D}$ such as the one in the introduction. That is, we assume $\varphi(t,z)=\varphi_t(z)$ is a $\mathbb{D} \times \mathbb{D^*} \to \mathbb{C}$ holomorphic motion which in addition is conformal in the parameter $z$. 
By the previous discussion, each $\varphi_t$ extends to a $|t|$-quasiconformal mapping of $\mathbb{C}$. Moreover, by symmetrising the Beltrami coefficients
like in  \cite{kuhnau,smirnov},
we see that  $\varphi_t(\mathbb{S}^1)$ is a $k$-quasicircle, where $|t|=2k/(1+k^2)$. More precisely, $\varphi_t(\mathbb{S}^1)=f(\mathbb{R\cup\{\infty\}})$ for a $k$-quasiconformal map $f \colon \mathbb{\hat C} \to \mathbb{\hat C}$ of the Riemann sphere $\mathbb{\hat C}$, which is antisymmetric with respect to the real line in the sense that 
\[\mu_f(z)=- \overline{\mu_{f}(\overline{z})} \quad \mbox{for a.e.~$z \in \mathbb{C}$}.
\]
Smirnov used this antisymmetric representation to prove \eqref{eq:smirnovdk}.
In terms of the conformal maps $\varphi_t$, Smirnov's result takes the form mentioned in \eqref{eq:smirnov-intro}.

\subsection{Heuristics for $\sigma^2(\mathcal{S}\mu) \leqslant 1$.}
An estimate based on the $\tau=2$ case of \cite[Theorem 3.3]{prause-smirnov}  tells us roughly that for $R>1$,
\begin{equation}
\label{eq:prause-smirnov}
 \frac{1}{2\pi R} \int_{|z| = R}  | \varphi'_t(z)|^2 |dz|  \; \leqslant \, C(|t|) \, (R-1)^{-|t|^2}.
\end{equation}
 (The precise statement is somewhat weaker but we are not going to use this.) 
A natural strategy for proving $\sigma^2(\mathcal{S}\mu) \leqslant 1$ is to consider the holomorphic motion of principal mappings $\varphi_t$ generated by $\mu$,
\[ \overline \partial \varphi_t = t \mu \, \partial \varphi_t, \quad t \in \mathbb{D}; \qquad \varphi_t(z) = z + \mathcal O (1/z) \quad \mbox{as } z \to \infty.
\]
For the derivatives, we have the Neumann series expansion:
\begin{equation} \label{neumann}
\varphi'_t = \partial  \varphi_t = 1+t \mathcal{S}\mu+t^2 \mathcal{S}\mu \mathcal{S}\mu+ \ldots, \qquad z \in \DD^*.
\end{equation}
In view of this, taking the limit $t \to 0$ in \eqref{eq:prause-smirnov}, one obtains a growth bound (as $R \to 1$) for the integrals $\int_{|z| = R}  |\mathcal{S}\mu|^2 |dz|$.
However, in order to validate this strategy, one needs to have good control on the constant term $C(|t|)$ in \eqref{eq:prause-smirnov}. Namely, one would need to show that $C(|t|) \to 1$ as $t \to 0$ fast enough, for instance at a quadratic rate $C(|t|) \leqslant C^{|t|^2}$. Unfortunately, while the growth exponent  in \eqref{eq:prause-smirnov} is effective, the constant is not. 

In order to make this strategy work, we need two improvements. First, we work with quasiconformal maps that are antisymmetric with respect to the unit circle; and secondly, we use normalised solutions instead of principal solutions. 
One of the key estimates will be Theorem \ref{intmeans11} which is the counterpart of \eqref{eq:prause-smirnov} for antisymmetric maps, but crucially with a multiplicative constant of the form $C(\delta)^{k^2}$.
This naturally complements the Hausdorff measure estimates of \cite{ptut}.

\subsection{Interpolation}

Let $(\Omega,\sigma)$ be a measure space and  $L^p (\Omega,\sigma)$ be the usual spaces of complex-valued $\sigma$-measurable functions on $\Omega$ 
equipped with the (quasi)norms
\[ \| \Phi\|_p  = \left(\int_\Omega |\Phi(x)|^p \;\textrm{d}\sigma(x)\right)^{\frac{1}{p}}\,,\;\;\;0< p<\infty.
\]
In the papers \cite{A} -- \cite{AIPS2},  holomorphic deformations were used to give sharp bounds on the distortion of quasiconformal mappings. In \cite{AIPS}, the method was formulated as a compact and general interpolation lemma:

\begin{lemma}\cite[Interpolation Lemma for the disk]{AIPS}\label{interpolation}
Let $0<p_0 ,p_1 \leqslant \infty$ and
$\;\{\Phi_{\lambda}\,;\; |\lambda| \,<\, 1\}\, \subset\,\mathscr M(\Omega, \sigma)\,$ be an analytic and non-vanishing family of measurable functions  defined on 
a domain $\Omega$. Suppose
$$
M_0\,:=\, \| \Phi_0\|_{p_0} <\infty,\;\;\;  M_1\,:=\, \underset{|\lambda| <1}{\sup} \| \Phi_\lambda\|_{p_1} <\infty \;\;\; \text{and }\; 
M_r \,:=\, \underset{|\lambda|=r}{\sup}\,\|\Phi_\lambda\|_{p_r},
$$
where
\[ 
 \frac{1}{p_r}=\frac{1-r}{1+r}\cdot \frac{1}{p_0} + \frac{2r}{1+r}\cdot \frac{1}{p_1}.\\
\]
\vskip8pt

\noindent Then, for every $\, 0\leqslant r < 1\,$,  we have
\begin{equation}
M_r \leqslant  \;M_0^{\frac{1-r}{1+r}}\cdot M_1^{\frac{2\,r}{1+r}}\;<\,\infty.
\end{equation}
\end{lemma}

To be precise, in the lemma we consider analytic families $\Phi_\lambda$ of measurable
functions in $\Omega$, i.e.~jointly measurable functions $ (x, \lambda) \mapsto \Phi_\lambda (x)$ defined on $\Omega \times \DD$, for which  there exists a set $E \subset\Omega$ of $\sigma$-measure zero such that for all $x \in\Omega\setminus E$, the map $\lambda \mapsto \Phi_\lambda(x)$ is analytic and non-vanishing in $\DD$.

For the study of the asymptotic variance of the Beurling transform, we need to combine interpolation with ideas from \cite{smirnov} to take into account the antisymmetric dependence on $\lambda$, see Proposition \ref{apupropo}. In this special setting,  Lemma \ref{interpolation} takes the following form:

\begin{corollary}\label{le:interpolation}
Suppose  $\;\{\Phi_{\lambda}\,;\, \lambda \in \DD\}$ is an analytic family of measurable functions, such that  for every $ \lambda \in \DD$, 
\begin{equation}
\label{eq:Phi-antisym} 
  \Phi_\lambda(x)\not=0 \; \mbox{and } \;  \big|\Phi_\lambda(x)\big| = \big| \Phi_{- \overline \lambda}(x)\big|, \quad \mbox{ for a.e.~}  x \in \Omega.
\end{equation}
Let $0 < p_0,p_1 \leqslant \infty$. Then, for all $0 \leqslant k <1$ and exponents $p_k$ defined by
 \[
 \frac{1}{p_k}=\frac{1-k^2}{1+k^2}\cdot \frac{1}{p_0} + \frac{2k^2}{1+k^2}\cdot \frac{1}{p_1},
\]
we have 
\[ \| \Phi_k \|_{p_k} \leqslant \| \Phi_0 \|_{p_0}^{\frac{1-k^2}{1+k^2}}  \left({\sup}_{\{ |\lambda| <1\} }  \| \Phi_\lambda \|_{p_1} \right)^{\frac{2k^2}{1+k^2}},
\]
assuming that the right hand side is finite.
\end{corollary}

\begin{proof}
Consider the analytic family $\lambda \mapsto \sqrt{\Phi_\lambda(x)\, \Phi_{-\lambda}(x)}$.
The non-vanishing condition ensures that we can take an analytic square-root. Since the dependence with respect to  $\lambda$ gives an even analytic function, there is a (single-valued) analytic family $\Psi_\lambda$ such that 
$$\Psi_{\lambda^2}(x)=\sqrt{\Phi_\lambda(x)\, \Phi_{-\lambda}(x)}.$$

Observe that $|\Phi_{\lambda}(x)|=|\Psi_{\lambda^2}(x)|$ for real $\lambda$ by the condition \eqref{eq:Phi-antisym}.
By the Cauchy-Schwarz inequality, $\Psi_\lambda$ satisfies the same $L^{p_1}$-bounds:
\[ \|\Psi_{\lambda^2}\|_{p_1} \, \leq \,  \|\Phi_\lambda\|_{p_1}^{1/2}\|\Phi_{-\lambda}\|_{p_1}^{1/2} \, \leq \, {\sup}_{\{ |\lambda| <1\} }  \| \Phi_\lambda \|_{p_1}, \qquad \lambda \in \DD.
\]
We can now apply the Interpolation Lemma for the non-vanishing family ${\Psi}_\lambda$ with $r=k^2$ to get
\begin{align*}
\| \Phi_k \|_{p_k} = \| {\Psi}_{k^2}\|_{p_k} 
 &\leqslant  \| {\Psi}_{0}\|_{p_0}^{\frac{1-k^2}{1+k^2}}\left({\sup}_{\{ |\lambda| <1\} }  \| {\Psi}_\lambda \|_{p_1} \right)^{\frac{2k^2}{1+k^2}}\\
&\leqslant \| \Phi_0 \|_{p_0}^{\frac{1-k^2}{1+k^2}}  \left({\sup}_{\{ |\lambda| <1\} }\| \Phi_\lambda \|_{p_1} \right)^{\frac{2k^2}{1+k^2}}.
\end{align*}
\end{proof}

\section{Upper bounds}  \label{sec:upper}

In this section, we apply quasiconformal methods for finding bounds on integral means to the problem of maximising the asymptotic variance  $\sigma^2(\mathcal{S}\mu)$  of  the Beurling transform. Our aim is to establish  the following result:
 \begin{theorem} \label{upperthm}
 Suppose $\mu$ is measurable with  $|\mu| \leqslant \chi_{\mathbb{D}}$.
  Then, for all $1 < R < 2$,
\begin{equation} \label{growth1}  \frac{1}{2\pi} \int_{0}^{2\pi} |{\mathcal S} \mu (Re^{i\theta})|^2  d\theta \leq (1+\delta) \log\frac{1}{R-1} + c(\delta), \qquad 0 < \delta < 1,
 \end{equation}
 where $c(\delta) < \infty$ is a constant depending only on $\delta$.
 \end{theorem}
 
 The growth rate in \eqref{growth1} is interesting only for $R$ close to $1$: For $|z| = R > 1$,  we always  have the pointwise bound
 \begin{equation} \label{growth13}   
 |{\mathcal S} \mu (z)| =  \left| \frac{1}{\pi} \int_{\DD}  \frac{\mu(\zeta)}{(\zeta - z)^2} dm(\zeta) \right| \leq  \frac{1}{(R-1)^2}.
 \end{equation}
  
\noindent  It is clear that Theorem \ref{upperthm} implies  $\Sigma^2 \leqslant 1$, i.e.~the statement from Theorem \ref{thm:upperbound-intro} that
\begin{equation} \label{goal4} 
 \sigma^2(\mathcal{S}\mu) = \frac{1}{2\pi}\limsup_{R\to1^+}\,  \frac{1}{|\log(R-1)|} \int_{0}^{2 \pi} |\mathcal{S} \mu (Re^{i \theta})|^2 \; d\theta \;  \leqslant \;   1
\end{equation}
whenever $|\mu| \leqslant \chi_{\mathbb{D}}$. 

The proof of Theorem \ref{growth1} is based on holomorphic motions and quasiconformal distortion estimates. In particular, we 
 make strong use of the ideas of Smirnov \cite{smirnov}, where he showed that the dimension of a $k$-quasicircle is at most $1+k^2$. We first  need a few preliminary results.

\subsection{Normalised solutions}  
The classical Cauchy transform of a function $\omega \in L^p(\C)$ is given by
\begin{equation} \label{cauch}
\mathcal{C}\omega(z)= \frac{1}{\pi} \int_{\C} \frac{\omega(\zeta)}{z-\zeta}dm(\zeta).
\end{equation}
For us it will be convenient to use a modified version
\begin{eqnarray}
\label{c}
 {\mathcal C}_1 \omega (z) &:= & \frac{1}{\pi} \int_{\C} \omega(\zeta) \left[\frac{1}{z-\zeta} - \frac{1}{1-\zeta}  \right] dm(\zeta) \\  %\nonumber \\ 
& = & (1-z)  \frac{1}{\pi} \int_{\C} \omega(\zeta) \frac{1}{(z-\zeta)(1-\zeta)}  dm(\zeta) \nonumber
\end{eqnarray}
defined pointwise for compactly supported functions $\omega \in L^p(\C)$, $p > 2$. Like the usual Cauchy transform, 
the modified Cauchy transform satisfies the identities $\overline \partial ( {\mathcal C}_1 \omega ) =  \omega$
and  $ \partial ( {\mathcal C}_1 \omega ) ={\mathcal S}  \omega$. Furthermore, ${\mathcal C}_1 \omega $ is  continuous, vanishes at $z=1$ and has  the asymptotics  
$${\mathcal C}_1 \omega (z) =  -\frac{1}{\pi} \int_{\C}  \frac{\omega(\zeta)}{1- \zeta} dm(\zeta) + {\mathcal O}(1/z) \quad \mbox{ as } \; z \to \infty.$$

We will consider quasiconformal mappings with Beltrami coefficient  $\mu$ supported on unions of annuli
$$A(\rho,R) := \{z \in \C : \rho < |z| < R \}.$$
Typically, we need to make sure that the support of the Beltrami coefficient is symmetric with respect to the reflection in the unit circle. Therefore, it is
convenient to use the notation
\begin{eqnarray}
A_R  & := & A(1/R,R), \qquad 1 < R < \infty \qquad \mbox{and} \\
A_{\rho,R} & := & A(1/R,1/\rho) \cup A(\rho,R), \qquad 1 < \rho < R < \infty. \label{kaksi}
\end{eqnarray}

 For coefficients supported on annuli $A_R$, the  normalised homeomorphic solutions to the 
  Beltrami equation
\begin{equation}
\label{b}
\overline \partial f(z) = \mu(z) \partial f(z) \quad \mbox{for a.e.~} z \in \C, \qquad f(0)=0, \; f(1)= 1,
\end{equation}
admit a simple representation:
\begin{proposition}\label{rep}
Suppose $\mu$ is supported on ${A_R}$ with $\| \mu \|_{\infty} < 1$ and $f \in W^{1,2}_{loc}(\C)$ is the normalised  homeomorphic solution to \eqref{b}. Then 
\begin{equation}
\label{c1}
 f(z) = z \exp( {\mathcal C}_1 \omega (z) ), \qquad z \in \C,
\end{equation}
where $\omega \in L^p(\C)$  for some $p > 2$, has support contained in $A_R$ and 
\begin{equation}
\label{beltr4} \omega(z) - \mu(z){\mathcal S}  \omega (z) = \frac{\mu(z)}{z} \quad \mbox{for a.e.~} z \in \C.
\end{equation}
\end{proposition}
\begin{proof}
First, if $\omega$ satisfies the above equation, then
 $$ \omega = (Id - \mu{\mathcal S})^{-1} \left( \frac{\mu(z)}{z} \right) =  \frac{\mu(z)}{z} +  \mu{\mathcal S}\left(\frac{\mu(z)}{z}\right) +  \mu{\mathcal S}\mu{\mathcal S}\left(\frac{\mu(z)}{z}\right) + \cdots
 $$
 with the series converging in $L^p(\C)$ whenever $\| \mu \|_{\infty} \|Ê\mathcal S \|_{L^p} < 1$, in particular for some $p>2$. The solution, unique in $L^p(\C)$,  clearly has support contained in $A_R$.

 If $f(z)$ is as in \eqref{c1}, then  $f \in W^{1,2}_{loc}(\C)$ and satisfies \eqref{b} with the required normalisation. To see that $f$ is a homeomorphism, note that 
 \begin{equation}
\label{pr1}
f(z) = \alpha [z + \beta   + {\mathcal O}(1/z)] \qquad \mbox{  as } \; z\to \infty,
\end{equation}
 where 
 \begin{equation}
\label{pr2}
 \alpha = \exp\left(-  \frac{1}{\pi} \int_{\C}  \frac{\omega(\zeta)}{1-\zeta} dm(\zeta) \right) \neq 0 \quad \mbox{and} \quad \beta =  \frac{1}{\pi} \int_{\C} \omega(\zeta) dm(\zeta)
\end{equation}
which shows that $f$ is a composition of a similarity and a  principal solution to the Beltrami equation. Since every principal solution to a Beltrami equation is automatically a homeomorphism \cite[p.169]{AIMb}, we see that $f$ must be a homeomorphism as well. 
The proposition now follows from the uniqueness of normalised homeomorphic solutions to \eqref{b}. 
\end{proof}

\subsection{Antisymmetric mappings}

If the Beltrami coefficient in  \eqref{b} satisfies $\overline{\mu(z)} = \mu(\overline z )$, then by the uniqueness of the normalised solutions,
we have $\overline{f(z)} = f(\overline z )$ and $f$ preserves the real axis.

For normalised solutions preserving the unit circle, the corresponding condition for $f$ is
$f(1/\overline z ) = 1/{\overline{f(z)}}$ which asks for the Beltrami coefficient to satisfy
$\mu(\frac{1}{\, \overline z \,} ) \frac{\overline z^2}{z^2} = {\overline{\mu(z)}}$ for a.e.~$z \in \C.$
In this case, we say that the Beltrami coefficient $\mu$ is {\it symmetric} (with respect to the unit circle).
Following \cite{smirnov}, we say that $\mu$ is {\it antisymmetric} if 
\begin{equation}
\label{belt}
\mu\left(\frac{1}{\, \overline z \,} \right) \frac{\overline z^2}{z^2} \; = \;   - \, {\overline{\mu(z)}} \qquad \mbox{for a.e.~} z \in \C.
 \end{equation}

\medskip

Given an antisymmetric $\mu$ supported on $A_R$ with $ \| \mu \|_{\infty} = 1 $, define 
$$ \mu_{\lambda}(z)  = \lambda \, \mu(z), \quad \lambda \in \DD, 
$$
and let $f_{\lambda}$ be the corresponding normalised homeomorphic solution to \eqref{b} with $\mu = \mu_{\lambda}$. 
It turns out that in case of  mappings  antisymmetric with respect to the circle, the expression 
$$  \Phi_\lambda(z): = z \,\frac{  \, \partial f_\lambda(z)\, }{f_\lambda(z)}
$$
has the proper invariance properties similar to those used in \cite{smirnov}:

\begin{proposition} \label{apupropo}
For every $\lambda \in \DD$ and $z \in \C$,
$$\frac{1}{\overline z} \,  \frac{ \, \partial f_\lambda(1/\overline z ) \, }{f_\lambda(1/\overline z)}  = 
\overline{\left[z  \frac{ \, \partial f_{(- \overline \lambda \,)}(z)\, }{f_{(-\overline{\lambda}\,)}(z)} \right]}.
$$
In particular, 
$$\left|  \frac{ \, \partial f_\lambda(z) \, }{f_\lambda(z)} \right|  = \left|  \frac{ \, \partial f_{ (- \overline \lambda \,) }(z) \, }{f_{(- \overline \lambda \,)}(z)} \right| \qquad \mbox{whenever } \;  |z| = 1.
$$ 
\end{proposition}
\begin{proof} Let 
 \begin{equation}
\label{identity2}
 g_\lambda(z) = \frac{1}{\; \overline{f_\lambda(1/\overline z)} \; }, \qquad z \in \C. 
\end{equation}
By direct calculation, $g_\lambda$ has complex dilatation $\overline{\, \lambda \mu(\frac{1}{\, \overline z \,} ) \frac{\overline z^2}{z^2}\, }$ 
 which by our assumption on antisymmetry is equal to $ - \overline{\lambda} \mu(z)$. Since $g_\lambda$ and $f_{- \overline{\lambda}}$ are normalised solutions to the same Beltrami equation, the functions must be identical.
 Differentiating the identity \eqref{identity2} with respect to $\partial/\partial z$, we get
 $$  
 \partial f_{(- \overline{\lambda})} (z) = \frac{1}{\overline z^2} \;  \frac{ \, \overline{ \partial f_{ \lambda}(1/\overline z) }\, }{ \; \overline{ f_{\lambda}(1/\overline z)^2 } \, } =  f_{(- \overline{\lambda})} (z) \frac{1}{\overline z^2} \;  \frac{ \, \overline{ \partial f_{ \lambda}(1/\overline z) }\, }{ \; \overline{ f_{\lambda}(1/\overline z) } \, }.
 $$
Rearranging and taking the complex conjugate gives the claim.
\end{proof}

\subsection{Integral means for antisymmetric mappings}

 For $1 < R < 2$, consider a quasiconformal mapping $f$ whose Beltrami coefficient is supported on $A_{R,2}$. 
 Since $f$ is conformal in the narrow annulus $\{ \frac{1}{R} < |z| < R\}$, it is reasonable  to study bounds for the integral means involving the derivatives of $f$ on the unit circle. We are especially interested in the dependence of these bounds on $R$ as $R \to 1^+$.

 \begin{theorem} \label{intmeans11}
 Suppose $\mu$ is  measurable,  $|\mu(z)| \leq (1-\delta) \chi_{A_{R,2}}(z)$, and that $\mu$ is antisymmetric. Let $0 \leq k \leq 1$.
 
  If  $f = f_k \in W^{1,2}_{loc}(\C)$ is the normalised homeomorphic solution to  $\overline \partial f(z) = k \mu(z) \partial f(z)$, then 
 \begin{equation}
 \label{eq:intmeans11}
  \frac{1}{2\pi} \int_{|z| = 1}  \left|  \frac{ \,  f'(z) \, }{f(z)} \right|^2 |dz|  \; \leq \, C(\delta)^{k^2} \, (R-1)^{-\frac{2k^2}{1+k^2}},
 \end{equation}
  where $C(\delta) < \infty$ is a constant depending only on $\delta$.
 \end{theorem}
  \bigskip
 
 \noindent The assumption \, $\|\mu(z)\|_{\infty} \leq 1-\delta$  \, above, where $\delta > 0$ is fixed but arbitrary, is made to guarantee that we have
  uniform bounds in (\ref{eq:intmeans11}) for all $k <1$. To estimate the asymptotic variance of the Beurling transform, we will study the nature of these  bounds as $k \to 0$,
 but we need to  keep  in mind the dependence on the auxiliary parameter $\delta > 0$.  
 
  \bigskip
  
\begin{proof}[Proof of Theorem \ref{intmeans11}]
  We embed $f$ in a holomorphic motion by setting
  \begin{equation*}
 \mu_\lambda(z) = \lambda \; \mu(z), \qquad \lambda \in \DD.
 \end{equation*}
 Let $f_\lambda$ denote the normalised solution to the Beltrami equation $f_{\overline z} = \mu_\lambda f_z$, with the representation \eqref{c1} described in  Proposition \ref{rep}.  The uniqueness of the solution implies that $f_k=f$.  
 
We now apply Corollary \ref{le:interpolation} to the family
   \begin{equation}
\label{perhe}
 \Phi_\lambda(z): = z \frac{ \, (f_\lambda)'(z) \, }{f_\lambda(z)}, \qquad \lambda \in \DD, \, z \in \mathbb S^1.
 \end{equation}
By \cite[Theorem 5.7.2]{AIMb},  
 the map   is well-defined,  nonzero and holomorphic in $\lambda$ for each 
 $z \in \mathbb S^1$.
 By the antisymmetry of the dilatation $\mu$, we can use Proposition \ref{apupropo} to get the identity 
  \begin{equation}
\label{perhe2}
 \big|\Phi_\lambda(z)\big| = \left|  \frac{ \, \partial f_\lambda (z) \, }{f_\lambda(z)} \right|  = \left|  \frac{ \, \partial f_{ (- \overline \lambda \,) }(z) \, }{f_{(- \overline \lambda \,)}(z)} \right|  = \big| \Phi_{- \overline \lambda}(z)\big|, \qquad z \in \mathbb S^1.
 \end{equation}

 We first find a global $L^2$-bound,  independent of $ \lambda \in \DD$. 
  For this purpose, we estimate %, 
  $$\frac{1}{2\pi}  \int_{A_R}  \left|  \frac{ \,  f_\lambda'(z) \, }{f_\lambda(z)} \right|^2 dm(z).
  $$
% independent of $\lambda \in \DD$. 
 Recall  that $1 < R < 2$ by assumption. Since all $f_\lambda$'s are  normalised $\frac{1+\delta}{1-\delta}$-quasiconformal mappings, we have 
 $$ |f_\lambda(z)|  = \frac{ |f_\lambda(z) - f_\lambda(0)| }{|f_\lambda(1)  - f_\lambda(0)| }\geq 1/\rho_\delta, \qquad 1/R < |z| < R,
$$
together with
$$ f_\lambda(A_R) \subset  f_\lambda B(0,2) \subset B(0,\rho_\delta). 
$$
 Therefore,
\begin{equation}
\label{maar}
\frac{1}{2\pi} \int_{A_R}  \left|  \frac{ \,  f_\lambda'(z) \, }{f_\lambda(z)} \right|^2 dm(z) \leq \frac{1}{2\pi} \rho_\delta^2 |  f_\lambda A_R| \leq  \rho_\delta^4/2
\end{equation}
for some constant $1 <  \rho_\delta < \infty$ depending only on $\delta$. In particular,
$$ (R-1) \, \frac{1}{2\pi} \int_{|z|=1}  \left|  \frac{ \,  f_{\lambda}'(z) \, }{f_{\lambda}(z)} \right|^2 |dz| \; \leqslant c(\delta) < \infty, \qquad \lambda \in \DD,
$$
where the bound $c(\delta)$ depends only on $0 < \delta < 1$.

We now use interpolation to improve the $L^2$-bounds near the origin. We choose $p_0=p_1=2$, 
$\Omega = \mathbb{S}^1$ and  $d\sigma(z) = \frac{R-1}{2\pi}|dz|$. Applying Corollary \ref{le:interpolation} gives
$$  (R-1) \, \frac{1}{2\pi} \int_{|z|=1}  \left|  \frac{ \,  f_{k}'(z) \, }{f_{k}(z)} \right|^2 |dz| \; \leqslant  \left(R-1 \right)^{\frac{1-k^2}{1+k^2}} c(\delta)^{\frac{2k^2}{1+k^2}},
$$
which implies Theorem \ref{intmeans11}. 
\end{proof}
 
 \subsection{Integral means for the Beurling transform}\label{sec.means}

We now use infinitesimal estimates for quasiconformal distortion to give
bounds for the integral means of ${\mathcal S} \mu$.
We begin with the following lemma:

  \begin{lemma} \label{intmeans3} Given $1 < R < 2$, suppose $\mu$ is an antisymmetric Beltrami coefficient with $\supp \mu \subset A_{R,2}$ andÊ$\; \| \mu \|_{\infty} \le 1$.  Then,  $ \mu_1(z) :=   \frac{\mu(z)}{z} \; %$|\mu(z)| \leq (1-\delta) \chi_{A_{R,2}}(z)$
  %$|\mu| \le \chi_{A_{R,2}}$.
  $
 %$ \mu_1(z) :=  \frac{\mu(z)}{z}$    %$|\mu(z)| \leq (1-\delta) \chi_{A_{R,2}}(z)$
  satisfies
  $$  \frac{1}{2\pi} \int_{|z| = 1} |{\mathcal S} \mu_1 (z)|^2 |dz| \leq (1+\delta) \log \frac{1}{(R-1)^2} +  \log C(\delta/4), \quad 0 < \delta < 1,
  $$
  where $C(\delta)$ is the constant from Theorem \ref{intmeans11}.
  \end{lemma} 

\begin{proof}
First, observe that if $h$ is any  $L^1$-function vanishing in the annulus $\{z: 1/R < |z| < R\}$, %$R>1$,
 by the theorems of Fubini and Cauchy,
    \begin{eqnarray*}
\label{nolla1}
 \frac{1}{2\pi} \int_{|z| = 1} z ({\mathcal S}h) (z) \, |dz|  &=&  \frac{1}{2 \pi i}\int_{{\mathbb S}^1} (\mathcal S h)(z) dz  \nonumber \\ &=& \frac{1}{\pi} \int_{\C} h(\zeta) \left[ \frac{1}{2 \pi i}\int_{{\mathbb S}^1} \frac{1}{(\zeta - z)^2} dz \right] dm(\zeta) = 0.
\end{eqnarray*}
%\\
To apply Theorem \ref{intmeans11}, take $0 < k < 1$ and solve the Beltrami equation $\overline \partial f(z) = k \nu(z) \partial f(z)$ for the coefficient   $\nu(z) = (1-\delta) \mu(z)$. Let $f_k \in W^{1,2}_{loc}(\C)$ be the normalised homeomorphic solution in $\C$.
 
 Recall from \eqref{c1} that  $f_k$ has the representation $f_k(z) = z \exp( {\mathcal C}_1 \omega (z) )$ 
 where
$$\omega =   (Id - k \, \nu{\mathcal S})^{-1} \left( \frac{k \,\nu(z)}{z} \right) 
 = k (1-\delta) \, \mu_1(z) + k^2 (1-\delta)^2 \,\nu{\mathcal S}\mu_1(z) + \cdots
 $$
and the series converges in $L^p(\C)$ for some fixed $p = p(\delta) > 2$.  From this representation, we see that
\begin{equation}
\label{ratio1}
z  \frac{ \,  f_k'(z) \, }{f_k(z)}  % = 1 + z {\mathcal S} \omega (z) 
= 1+ k (1-\delta) z {\mathcal S} \mu_1(z) + k^2 (1-\delta)^2 z {\mathcal S}\nu{\mathcal S}\mu_1(z) + \mathcal O(k^3)
\end{equation}
holds pointwise in the  annulus $\{z: 1/R < |z| < R\}$, where $\nu$ and $\omega$ vanish.
It follows that 
%In particular, we see that the integral means satisfy
\begin{equation}
\label{intmeans2}
  \frac{1}{2\pi} \int_{|z| = 1}  \left| \frac{  f_k'(z)}{f_k(z)} \right|^2 |dz| = 1 + k^2  (1-\delta)^2 \frac{1}{2\pi} \int_{|z| = 1} |{\mathcal S} \mu_1 (z)|^2 |dz|+  \mathcal O(k^3).
\end{equation}
 Finally, combining \eqref{intmeans2} with Theorem \ref{intmeans11}, we obtain
 $$ 
 \hspace{-2cm}1 + k^2 (1-\delta)^2  \frac{1}{2\pi} \int_{|z| = 1} |{\mathcal S} \mu_1 (z)|^2 |dz|  +  \mathcal O(k^3)
 $$
 $$
\;  \leq \;  \exp\left(  k^2 \log C(\delta)  + \frac{k^2}{1+k^2} \log \frac{1}{(R-1)^2}\right)$$
$$ \hspace{4cm}  = 1 + k^2 \log  C(\delta)  + k^2 \log \frac{1}{(R-1)^2}  +  \mathcal O(k^4).
 $$
Taking $k \to 0$, we find that
$$\frac{1}{2\pi} \int_{|z| = 1} |{\mathcal S} \mu_1 (z)|^2 |dz|  \leq (1-\delta)^{-2} \log \frac{1}{(R-1)^2} +  (1-\delta)^{-2} \log  C(\delta).
$$
As $ (1-\delta/4)^{-2} \leq 1+\delta$,  replacing $\delta$ by $\delta/4$
proves the lemma.     \end{proof}

 \begin{corollary} \label{means4} Given $1 < R < 2$, 
  suppose $\mu$ is a Beltrami coefficient with $\supp \mu \subset A(1/2,1/R)$ and  $\| \mu \|_{\infty} \leq 1$. Then,
   $$  \frac{1}{2\pi} \int_{|z| = 1} |{\mathcal S} \mu(z)|^2 |dz| \leq (1+\delta) \log \frac{1}{(R-1)} +   \frac{1}{2} \log  C(\delta/4), \quad 0 < \delta < 1,
  $$
  where $C(\delta)$ is the constant from Theorem \ref{intmeans11}.
  \end{corollary} 
\begin{proof} Define an auxiliary Beltrami coefficient $\nu$ by requiring $\nu(z)= z \mu(z)$ for $|z| \leq 1$ and $\nu(z)= -  \frac{z^2}{\,\overline z^2\,} \overline{\, \nu(1/\overline z) \,}$ for $|z| \geq 1$.
Then $\nu$ is supported on $A_{R,2}$, $\| \nu \|_{\infty} \leq 1$ and $\nu$ is antisymmetric, so that
with help of Lemma \ref{intmeans3} we can estimate the integral means of ${\mathcal S} \nu_1 $,
where $ \nu_1(z) = \frac{\nu(z)}{z}$.

On the other hand, the antisymmetry condition \eqref{belt} implies 
$$\mathcal C (\chi_{\DD} \nu_1)(1/\overline z) =  \overline{ \, \mathcal C( \chi_{\C \setminus \DD}\nu_1) (z)\,} - 
\overline{ \, \mathcal C( \chi_{\C \setminus \DD}\nu_1) (0)\,}$$
    for the Cauchy transform. Differentiating this with respect to $\partial/\partial \overline{z}$ gives
    $$\frac{1}{\, \overline z \,}  {\mathcal S} (\chi_{\DD} \nu_1)\left(\frac{1}{\, \overline z \,} \right) \, = - \; \overline{z  {\mathcal S} ( \chi_{\C \setminus \DD}\nu_1)(z)}. 
    $$
In particular, for $z$ on the unit circle  $\mathbb S^1$,
\begin{eqnarray*}
z  {\mathcal S} ( \nu_1)(z) & = &  z  {\mathcal S} ( \chi_{ \DD}\nu_1)(z) + z  {\mathcal S} ( \chi_{\C \setminus \DD}\nu_1)(z) \\
& = &  2i  \im \bigl[ z \, {\mathcal S} ( \chi_{ \DD}\nu_1)(z) \bigr] \\
& = & 2i \im \bigl[ z \, ({\mathcal S} \mu) (z) \bigr].
\end{eqnarray*}
In other words, the estimates of Lemma \ref{intmeans3} take the form
  $$  \frac{1}{2\pi} \int_{|z| = 1} \Bigl  |\im \bigl[ z \, ({\mathcal S} \mu) (z) \bigr] \Bigr |^2 |dz|  =  \frac{1}{4} \, \frac{1}{2\pi} \int_{|z| = 1} |{\mathcal S} \nu_1(z)|^2 |dz| $$
  $$\hspace{2cm} \leq \frac{1}{4} (1+\delta) \log \frac{1}{(R-1)^2} +  \frac{1}{4} \log C(\delta/4), \quad 0 < \delta < 1.
  $$
  By replacing $\mu$ with $i\mu$, we see that the same bound holds for the integral means of $\re \bigl[ z \, ({\mathcal S} \mu)(z)]$. Therefore,
 \begin{eqnarray*} 
 \frac{1}{2\pi} \int_{|z| = 1} |{\mathcal S} \mu(z)|^2 |dz| & = &
    \frac{1}{2\pi} \int_{|z| = 1} \Bigl  |\re \bigl[ z \, ({\mathcal S} \mu) (z) \bigr] \Bigr |^2 +  \Bigl  |\im \bigl[ z \, ({\mathcal S} \mu) (z) \bigr] \Bigr |^2 \;  |dz| \\
& \leq & (1+\delta) \log \frac{1}{R-1} +  \frac{1}{2} \log C(\delta/4) 
\end{eqnarray*}
for every $0 < \delta < 1$.
\end{proof}
 
 \subsection{Asymptotic variance}\label{variance}

 With these preparations, we are ready to prove Theorem \ref{upperthm}. We need to show that if  $\mu$ is measurable with $|\mu(z)| \leq \chi_{\DD}$, then for all $1 < R < 2$,
   $$  \frac{1}{2\pi} \int_{0}^{2\pi} |{\mathcal S} \mu (Re^{i\theta})|^2  d\theta \leq (1+\delta) \log\frac{1}{R-1} + c(\delta), \qquad 0 < \delta < 1,
   %(1+\delta) \, \bigl|\log (R-1)\bigr|  \; + \frac{1}{2} \log C(\delta/2) + 20, \qquad 0 < \delta < 1.
  $$
 % The constant $C(\delta) < \infty$ is  the same as in Theorem \ref{thmsmr} and Corollary \ref{intmeans1}. In particular, it depends only on $\delta$.
 where $c(\delta) < \infty$ is a constant depending only on $\delta$.

 \begin{proof}[Proof of Theorem \ref{upperthm}]
 For a proof of this inequality, first assume that additionally
   \begin{equation}
\label{apu3}
\mu(z) = 0 \quad \mbox{ for } |z| < 3/4; \qquad 1 < R < \frac{3}{2}.
\end{equation}
Then $ \nu(z) := \mu(Rz) $
   has support contained in $ B(0,1/R) \setminus B(0,1/2)$ so that we may apply Corollary \ref{means4}. Since $\mathcal S \nu(z) = \mathcal S\mu (Rz)$,
   $$   \frac{1}{2\pi} \int_{0}^{2\pi} |{\mathcal S} \mu (Re^{i\theta})|^2  d\theta =  \frac{1}{2\pi} \int_{|z| = 1} |{\mathcal S} \nu(z)|^2 |dz|
   $$
   $$\hspace{3cm} \leq (1+\delta) \log \frac{1}{R-1} +   \frac{1}{2} \log  C(\delta/4), 
   %\quad 0 < \delta < 1 < R < \frac{3}{2}
   $$
which is the desired estimate.

   \medskip
   
   For the general case when \eqref{apu3} does not hold, write $\mu = \mu_1 + \mu_2$ where $\mu_2(z) = \chi_{B(0,3/4)} \mu(z)$.
   As
\begin{equation*}
|\mathcal S \mu_2(z)| \leq \int_{\frac{1}{4} < |z-\zeta| < 2} \;  \frac{1}{|\zeta -z|^2} dm(\zeta) = 2 \pi \log(8), \qquad |z| = 1,
\end{equation*}
we have   
   $$  \hspace{-4cm} \frac{1}{2\pi} \int_{0}^{2\pi} |{\mathcal S} \mu_1 (Re^{i\theta}) + {\mathcal S} \mu_2 (Re^{i\theta})|^2  d\theta
   $$
   $$\leq (1+\delta)  \frac{1}{2\pi} \int_{0}^{2\pi} |{\mathcal S} \mu_1 (Re^{i\theta})|^2  d\theta + \left(1 + \frac{1}{\delta}\right)  \frac{1}{2\pi} \int_{0}^{2\pi} |{\mathcal S} \mu_2 (Re^{i\theta})|^2  d\theta $$
   $$ \leq (1+\delta)^2 \log \frac{1}{R-1} +  \frac{1+ \delta}{2} \log  C(\delta/4) +  \frac{1+ \delta}{\delta}  4 \pi^2 \log^2(8)
   $$
 for $ 0 < \delta < 1$ and $1 < R < \frac{3}{2}$; while for $R \ge \frac{3}{2}$, we have the pointwise bound \eqref{growth13}. Finally, replacing $\delta$ by $\delta/3$, we get the estimate in the required form,  thus proving the theorem.
 \end{proof}

\section{Lower bounds} \label{se:lowerbound}

Consider the family of polynomials
\[ P_t(z)=z^d+t\, z, \qquad |t| < 1,
\]
for $d \geqslant 2$. %near $0$. 
According to  \cite[Theorem 1.8]{mcmullen} or \cite{amo,ruelle}, the Hausdorff dimensions of their Julia sets satisfy
\begin{equation} \label{eq:dimjulia} 
\Hdim \, \mathcal J(P_t)= 1+ \frac{|t|^2 (d-1)^2}{4d^2 \log d} +\mathcal{O}(|t|^3).
\end{equation}
Moreover, each Julia set $\mathcal J(P_t)$ is a quasicircle, the image of the unit circle by a quasiconformal mapping of the plane.
 A quick way to see this is to observe that the immediate basin of attraction of the origin contains all the (finite) critical points of $P_t$. (From general principles, it is clear that the
 basin
 must contain at least one critical point, but by the $(d-1)$-fold
  symmetry of $P_t$, it must contain them all.)
  
 If $A_{P_t}(\infty)$ denotes the basin of attraction of infinity, for each $|t|< 1$ there is a canonical conformal mapping 
\begin{equation} \label{conju}
\varphi_t: \mathbb{D}^* = A_{P_0}(\infty) \to A_{P_t}(\infty)
\end{equation}
conjugating the dynamics: 
\begin{equation} \label{conju2} 
\varphi_t \circ P_0(z)=P_t \circ \varphi_t(z), \qquad z \in \mathbb{D}^*.
\end{equation}
 By Slodkowski's extended $\lambda$-lemma \cite{Slod} and the  properties of holomorphic motions, 
 $\varphi_t$  extends to a $|t|$-quasiconformal mapping of the plane, see e.g.~\cite[Section 12.3]{AIMb}.
 In particular, the extension maps the unit circle onto the Julia set ${\mathcal J}(P_t)$.

While the extensions given by the $\lambda$-lemma are natural, surprisingly it turns out that the maps $\varphi_t$ have  extensions with considerably smaller quasiconformal distortion, smaller by a factor of
\begin{equation} \label{const} 
c_d :=  \frac{d^{\frac{1}{d-1}}}{2}, \qquad 2 \leq d \in \N,  
\end{equation}
when $|t| \to 0$.

\begin{theorem} \label{improvement}
Let $P_t(z) = z^d + tz$ with $|t| < 1$. Then the canonical conjugacy  $\varphi_t: \mathbb{D}^*  \to A_{P_t}(\infty)$, defined in \eqref{conju}, has a $\mu_t$-quasiconformal extension with
\[ \|Ê\mu_t\|_{\infty} =  c_d |t| + \mathcal{O}(|t|^2).
\]
\end{theorem}
Here $c_2 = 1$, but $c_d < 1$ for $d \ge 3$. Hence for every degree $ \ge 3$ we have an improved bound for the distortion. Furthermore, when representing $\mathcal J(P_t)$ as the image of the unit circle by a map with as small distortion as possible, 
one can apply Theorem \ref{improvement} together with the symmetrisation method described in Section \ref{subsec:quasicircles} 
 to show that  each $\mathcal J(P_t)$ is a $k(t)$-quasicircle, where $$k(t) = \frac{c_d}{2} |t|  + \mathcal{O}(|t|^2).$$ 
%Smirnov's decomposition \cite{smirnov} for the conformal maps $\varphi_t$ and combine the above theorem with \eqref{eq:dimjulia}. 
By the dimension formula \eqref{eq:dimjulia}, % gives quasicircles with large Hausdorff dimension:
\begin{equation}
\label{eq:numericalvalues}
\Hdim \, \mathcal J(P_t)= 1+ \frac{ 4 d^{\frac{2}{1-d}}(d-1)^2}{d^2 \log d} \, |k(t)|^2 +\mathcal{O}(|k(t)|^3).
\end{equation}
In particular, when $d = 20$, we get $k$-quasicircles whose Hausdorff dimension is greater than $1+0.87913 \, k^2$, for small values of $k$. Therefore, Theorem \ref{julia} follows from Theorem \ref{improvement}.

The numerical values for the second order term of \eqref{eq:numericalvalues} are presented in Table \ref{table:comparison} below. These provide lower bounds on the asymptotic variance (or equivalently, on the quasicircle dimension asymptotics).
For comparison, we also show the values for the second order term of \eqref{eq:dimjulia} which correspond to the estimate 
on quasiconformal distortion provided by the $\lambda$-lemma.
Note that the first explicit lower bound on quasicircle dimension asymptotics  \cite{BP} is exactly the degree $2$ case of the upper-left corner.

\smallskip
\begin{table}[htbp]
\renewcommand{\arraystretch}{1.2}
\begin{center}
\begin{tabular}{|c|c|c|}
\hline
\; Degree \; & \; $\lambda$-lemma  \; & \; Bounds from \eqref{eq:numericalvalues} \; \\ \hline 
$d=2$ & $0.3606\ldots$ & $0.3606\ldots$ \\  \hline
$d=3$ & $0.4045\ldots$ & $0.5394\ldots$ \\ \hline
$d=4$ & $0.4057\ldots$ & $0.6441\ldots$ \\ \hline
$d=20$ & $0.3012\ldots$ & $0.8791\ldots$ \\ \hline
\end{tabular}
\end{center}
\bigskip
\caption{Comparison of lower bounds for $\Sigma^2$}
\label{table:comparison}
\end{table}

For the proof of Theorem \ref{improvement}, we find an improved representation for the infinitesimal vector field determined by $\varphi_t$.
Differentiating  \eqref{conju2}, we get a functional equation
\begin{equation}
\label{eq:functionaleq}
v(z^d)=d\, z^{d-1} v(z)+z
\end{equation}
 for the vector field 
$v=\frac{d\varphi_t}{dt}\big|_{t=0}$, 
which in turn forces the lacunary series expansion, see \cite[Section 5]{mcmullen}, 
\begin{equation} \label{eq:lacunaryjulia}
v(z)=-\frac{z}{d}\sum_{n=0}^\infty \frac{z^{-(d-1)d^n}}{d^n}, \qquad |z|>1.
\end{equation}

Our aim is to represent the lacunary series \eqref{eq:lacunaryjulia} as the Cauchy transform (or $v'$ as the Beurling transform) of an explicit bounded function supported on the unit disk. %Moreover, we want to minimise its $L^\infty$ bound. 
We will achieve this through the functional equation \eqref{eq:functionaleq}.
For this reason, we will look for Beltrami coefficients with invariance properties under $f(z) = z^d$, 
requiring that $f^*\mu = \mu$ in some neighbourhood of the unit circle, where
\begin{equation}(f^*\mu) (z) := \mu(f(z)) \frac{\overline{f'(z)}}{f'(z)}.
\end{equation}

%We record the transformation rule for the Cauchy transform under the pullback operation.
We first observe that the Cauchy transform \eqref{cauch} behaves similarly to a vector field under the pullback operation:
\begin{lemma}
\label{pullback-coefficients}
 Suppose $\mu$ is a Beltrami coefficient supported on the unit disk. Then,
\begin{equation}
\label{eq:pullback1}
 \frac{1}{dz^{d-1}} \Bigl \{ \mathcal  C\mu(z^d)- \mathcal C\mu(0) \Bigr\} = \mathcal C\Bigl ((z^d)^*\mu \Bigr )(z), \qquad z \in \mathbb{C}.
\end{equation}
\end{lemma}

\begin{proof} From   \cite[p. 115]{AIMb}, it follows that the Cauchy transform of a bounded, compactly supported function belongs to all  H\"older classes
 $\Lip_{\alpha}$ with exponents $0 < \alpha < 1$.
In particular, near the origin, the left hand side of (\ref{eq:pullback1}) is $\mathcal{O}(|z|^{1-\varepsilon})$ for every $\varepsilon > 0$.
  This implies that  the two quantities  in \eqref{eq:pullback1} have the same $(\partial/\partial \overline{z})$-distributional derivatives. As 
both vanish at infinity, they must be identically equal on the Riemann sphere.
\end{proof}

\begin{remark} \label{rem:cauchy0}
Since the left hand side in \eqref{eq:pullback1} vanishes at $0$, we always have $\mathcal C\bigl ((z^d)^*\mu \bigr )(0) = 0$. 
This can also be seen by using the change of variables $z \to \zeta \cdot z$ where $\zeta$ is a $d$-th root of unity.
\end{remark}

We will use the following basic Beltrami coefficients as building blocks:
 
\begin{lemma}
\label{lemma:basic-coefficient}
Let $\mu_n(z) := \bigl( {\overline z} /|z| \bigr)^{n-2} \chi_{A(r,\rho)}$ with $0 < r < \rho < 1$ and $2 \le n \in \mathbb{N}$. Then %$ \mathcal{C}\mu_n (0)=0 $ 
\begin{equation*}
 \mathcal{C}\mu_n (z) = \frac{2}{n} \left( \rho^{n} - r^{n} \right) z^{-(n-1)}, \qquad |z| > 1,
\end{equation*}
and $ \mathcal{C}\mu_n (0)=0 $.
\end{lemma}

\begin{proof} We compute:
\begin{equation*}
\int_{\mathbb{D}} \mu_n(w) \cdot w^{n-2} dm(w) = \int_{A(r,\rho)} |w|^{n-2} dm(w) = \frac{2\pi}{n} (\rho^{n} - r^{n}).
\end{equation*}
Hence, by orthogonality
\begin{eqnarray*}
\mathcal C\mu_n(z)
& = & \frac{1}{\pi z} \int_{\mathbb{D}}\frac{\mu_n(w)dm(w)}{(1-w/z)}\\
& = &  \frac{1}{\pi z} \sum_{j=0}^\infty z^{-j}\int_{\mathbb{D}}\mu_n(w)w^j dm(w)\\
 & = &  \frac{1}{\pi z} \cdot z^{-(n-2)} \cdot \frac{2\pi}{n} \cdot (\rho^{n} - r^{n}) \\
& = & \frac{2}{n} \cdot z^{-(n-1)} \cdot (\rho^{n} - r^{n})
\end{eqnarray*}
as desired. The claim  $ \mathcal{C}\mu_n (0)=0 $  follows similarly.
\end{proof}

To represent  power series in $z^{-1}$, we sum up $\mu_n$'s supported on disjoint annuli:

\begin{lemma} \label{mu}
For $d \ge 3$  and $\rho_0 \in (0,1)$, let 
$$n_j = (d-1) \, d^j, \qquad r_j = \rho_0^{1/n_j},   \qquad j = 0,1,2,\dots
$$
and define the Beltrami coefficient $\mu$ by
$$ \mu(z) =  \bigl( {\overline z} /|z| \bigr)^{n_j-2}, \qquad  r_j < |z| < r_{j +1}, \qquad  j \in \mathbb N, 
$$
while for $|z| < \rho_0^{1/n_0}$ and for $|z| > 1$, we set $\mu(z) = 0$.
With these choices,
\medskip

{\em (i)} %$\mu$ is $z^d$-invariant,  i.e. 
$\mu = (z^d)^* \mu+\mu \cdot \chi_{A(r_0,r_1)}$ and 

{\em (ii)} $\mathcal C\mu(z^d)=dz^{d-1}\, \mathcal C\mu(z) -\frac{2d}{d-1}[\rho_0^{1/d} - \rho_0] \cdot z, \quad |z|>1$.

\noindent In particular, for $|z| > 1$ we have

{\em (iii)} $\mathcal C\mu (z)=   -\frac{2d}{d-1} [\rho_0^{1/d} - \rho_0] \, v(z), %\quad |z| > 1,
$\quad  with 

\smallskip

{\em (iv)} 
$\mathcal S\mu (z) = -\frac{2d}{d-1} [\rho_0^{1/d} - \rho_0] \, v'(z),$ % \quad |z| > 1,$
\medskip

\noindent where $v = v_d$ is the lacunary series in \eqref{eq:lacunaryjulia}.
\end{lemma}

\begin{proof}
Claim (i) is clear from the construction. Inserting (i) into \eqref{eq:pullback1} and using   
%For (ii), observe that in view of (i) and 
Lemma \ref{lemma:basic-coefficient} gives (ii).
%, $\mathcal C\mu$ satisfies a functional equation. By Lemma \ref{lemma:basic-coefficient}, it is of the form (ii).
 This agrees with the functional equation (\ref{eq:functionaleq}) up to a constant term in front of $z$ which leads to (iii). Finally, (iv) follows by differentiation.
\end{proof}

\begin{remark}
The $d=2$ case of Lemma \ref{mu} is somewhat different since the vector field $v_2$ does not vanish at infinity, so
$v_2$ is not the Cauchy transform of any Beltrami coefficient. With the choice $n_j=2^{j+1}$, (ii) and (iii) hold up to an additive constant, while (iv) holds true as stated.
\end{remark}

Differentiating \eqref{eq:lacunaryjulia}, we see that
\begin{eqnarray*}
v'(z) & = & \sum_{n \ge 0}   z^{-(d-1)d^n} \cdot \frac{(d-1)d^n-1}{d^{n+1}} \\
& = & \frac{(d-1)}{d} \cdot  \sum_{n \ge 0}   z^{-(d-1)d^n} + b_0
\end{eqnarray*}
for some function $b_0 \in \mathcal B_0^*$, which implies %the asymptotic variance is
\begin{equation*}
\sigma^2(v'(z)) = \frac{(d-1)^2}{d^2 \log d}.
\end{equation*}
Therefore, the Beltrami coefficient $\mu=\mu_d$ from Lemma \ref{mu} satisfies
$$ \sigma^2(\mathcal{S}\mu) = \frac{4[\rho_0^{1/d}-\rho_0]^2}{\log d}.
$$

\noindent Fixing $d$ and optimising over $\rho_0 \in (0,1)$, simple calculus reveals that the maximum is obtained when $\rho_0=d^{\frac{d}{1-d}}.$ For this choice of $\rho_0$, %we note that
\begin{equation}  \label{better} v'(z) = - c_d \, \mathcal S\mu (z)  
\end{equation}
where $c_d$ is the constant from \eqref{const}. Moreover,
\begin{equation}
\sigma^2(\mathcal{S}\mu)  =  4 d^{\frac{2}{1-d}}  \frac{(d-1)^2}{d^2 \log d}
\end{equation}
obtains its maximum (over the natural numbers) at $d=20$, in which case %.For this degree $d$ we conclude
\begin{equation*}
\sigma^2(\mathcal{S}\mu_{20})>0.87913,  \qquad {\rm with} \; \;  |\mu| =  \chi_{\mathbb{D}}.
\end{equation*}
This construction proves Theorem \ref{goal5}. One can proceed further from these infinitesimal bounds and use \eqref{better} to produce quasicircles with large dimension.  This takes us to Theorem \ref{improvement}.

\begin{proof}[Proof of Theorem \ref{improvement}]
 By the extended $\lambda$-lemma, % the  B\"ottcher coordinates, 
 the conformal maps
$$
\varphi_t: \mathbb{D}^* \to A_{P_t}(\infty),
$$
admit quasiconformal extensions
$H_t: \C \to \C$, which depend holomorphically on $t \in \DD$. Since the Beltrami coefficient  $\mu_{H_t}$ is a holomorphic $L^\infty$-valued function of $t$,
 the vector-valued Schwarz lemma implies that
$$
\mu_{ H_t} = t\mu_0 +\mathcal{O}(t^2)
$$
for some Beltrami coefficient $|\mu_0| \leq \chi_{\mathbb{D}}$.  By developing $\varphi_t' = \partial_z   H_t$ as a Neumann series in $\mathcal{S}\mu_{H_t}$, c.f. \eqref{neumann},
%, see e.g. \cite[p. 165]{AIMb}
we get
 $$
 \mathcal{S}\mu_0(z) = v'(z), \qquad z \in \DD^*, 
 $$
 for the infinitesimal vector field $v=\frac{d\varphi_t}{dt}\big|_{t=0}$.
  
  On the other hand, if $\mu_d$ is the Beltrami coefficient  from Lemma \ref{mu}, it follows from \eqref{better} that  $\mu_{d}^\# := -c_d \, \mu_d$ 
also   satisfies $ \mathcal{S}\mu_{d}^\#(z) = v'(z) $ in $ \DD^*$.
%Now we can  replace $ H_t$ with $ H_t \circ N_t$, where $N_t$ represents a trivial element of
%the universal Teichm\"uller space, i.e. $N_t$ is a quasiconformal map which is the identity on the exterior unit disk.
Then the Beltrami coefficient $\mu_0 - \mu_{d}^\# $ is infinitesimally trivial,  and by \cite[Lemma V.7.1]{lehto}, we can find quasiconformal maps $N_t$ 
which are the identity on the exterior unit disk and have dilatations 
$\mu_{N_t} = t(\mu_0 - \mu_{d}^\#) + \mathcal{O}(t^2)$, $|t| < 1$. 
Therefore, we can  replace $ H_t$ with $ H_t \circ N_t^{-1}$ to obtain an extension of $\varphi_t$ with dilatation 
\begin{equation}
\label{eq:correctedextension}
\mu_{ H_t \circ N_t^{-1}} = t \mu_{d}^\#+ \mathcal{O}(t^2)
\end{equation}
as desired. This concludes the proof. \end{proof}

\begin{remark} (i) One can show that for $d \ge 2$, the Beltrami coefficient $\mu_d^\#$ constructed in Lemma \ref{mu} is not {\em infinitesimally extremal} which implies
that the conformal maps $\varphi_t$ (with $t$ close to 0) admit even more efficient extensions (i.e.~with smaller dilatations).
 One reason
to suspect that this may be the case is that $\mu_d^\#$ is not of the form $\frac{\overline{q}}{|q|}$ for some holomorphic quadratic differential $q$ on the unit disk; however, this fact
alone is insufficient.  
It would be interesting to find the dilatation of the most efficient extension, but this may be a difficult problem. %  as extremality in the universal Teichm\"uller space is rather delicate. 
For more on Teichm\"uller extremality, we refer the reader to
the survey of Reich \cite{reich}.

(ii) Let $M_{\shell}$ be the class of Beltrami coefficients of the form
$$
\sum_{j = 0}^\infty \ \biggl (\frac{\overline{z}}{|z|} \biggr )^{n_j-2} \cdot \chi_{A(r_i,r_{i+1})}, \qquad 0 \le r_0 < r_1 < r_2 < \dots < 1.
$$
One can show that 
$$\Sigma^2 > \sup_{\mu \in M_{\shell}} \sigma^2(\mathcal S \mu) =   \max_{d > 1}   4 d^{\frac{2}{1-d}}  \frac{(d-1)^2}{d^2 \log d} \approx
 0.87914$$
where the maximum is taken over all {\em real} $d > 1$.
\end{remark}

\section{Fractal approximation} \label{sec:fractal}
  
In this section, 
we present an alternative route to the upper bound for the asymptotic variance of the Beurling transform using (infinitesimal) fractal approximation.
We show that in order to compute
$
\Sigma^2 = \sup_{|\mu| \le \chi_{\mathbb{D}}} \sigma^2(\mathcal S\mu),
$
it suffices to take the supremum only over certain classes of
 ``dynamical'' Beltrami coefficients $\mu$ for which
McMullen's formula holds, i.e.
\begin{equation}
\label{eq:mcm-formula}
2 \, \frac{d^2}{dt^2} \biggl |_{t=0} \Hdim \varphi_t(\mathbb{S}^1) \ = \  
\lim_{R \to 1^+} \frac{1}{2\pi |\log(R-1)|}\int_{0}^{2\pi} |v_\mu'(Re^{i\theta})|^2 d\theta
 \end{equation}
where $\varphi_t$ is the unique principal homeomorphic solution to the Beltrami equation 
$
 \overline \partial \varphi_t = t \mu \, \partial \varphi_t
$
 and $v_\mu := \frac{d\varphi_t}{dt}\big|_{t=0}$ is the associated vector field. 
 By using the principal solution, we guarantee that $v_\mu$ vanishes at infinity which
 implies that $v_\mu = \mathcal C\mu$. We will use this identity repeatedly. (In general, when $\varphi_t$ does not necessarily fix $\infty$, $v_\mu$ and $\mathcal C\mu$ may differ by a quadratic polynomial $Az^2 + Bz + C$.)
 
Consider the following classes of dynamical Beltrami coefficients, with each subsequent class being a subclass of the previous one:

\smallskip
 
 \begin{itemize}
\item $M_{\bl} = \bigcup_f M_f(\mathbb{D})$ consists of Beltrami coefficients that are \emph{eventually-invariant} under some finite
\emph{Blaschke product} $f(z) = z \prod_{i=1}^{d-1} \frac{z-a_i}{1-\overline{a_i}z}$, i.e.~Beltrami coefficients which satisfy $f^*\mu = \mu$ in some open neighbourhood of the unit circle.

\item $M_{\ei} = \bigcup_{d \ge 2} M_{\ei}(d)$
  consists of Beltrami coefficients that are {eventually-invariant} under $z \to z^d$ for some $d \ge 2$.
  %, i.e.~Beltrami coefficients which satisfy $(z^d)^*\mu = \mu$ in some open neighbourhood of the unit circle.
\item $M_{\pp} = 
\bigcup_{d \ge 2}M_{\pp}(d)$ consists  of  $\mu \in M_{\ei}$ for which
 $v_\mu$ arises as the vector field
 associated to some \emph{polynomial perturbation} of $z \to z^d$, again for some $d \ge 2$.
For details, see Section \ref{subsection:pp}.
\end{itemize}

\begin{theorem}\cite{mcmullen}
\label{thm:mcmullen61}
If $\mu$ belongs to $M_{\bl}$, then the function $t \to \Hdim \varphi_t(\mathbb{S}^1)$ is real-analytic and \eqref{eq:mcm-formula} holds.
\end{theorem}

While McMullen did not explicitly state the relation between Hausdorff dimension and asymptotic variance for
 $M_{\bl}$, the argument in \cite{mcmullen} does apply to conjugacies % $w^{t\mu} \circ z^d \circ (w^{t\mu})^{-1}$
 $\varphi_t$  % \circ f \circ \varphi_t^{-1}$
  induced by this class of coefficients. 
%Of course, the class $M_{\pp}$ is explicitly covered in \cite{mcmullen}. 
Note that the class of polynomial perturbations is explicitly covered in McMullen's work, see \cite[Section 5]{mcmullen}.
We show:
 \begin{theorem} %The following equalities hold:
 \label{fractal-approximation}
\begin{equation*}
\Sigma^2 = \sup_{\mu \in M_{\ei}, \ |\mu| \le \chi_{\mathbb{D}}} \sigma^2(\mathcal S\mu)  
 = \sup_{\mu \in M_{\pp},\ |\mu| \le \chi_{\mathbb{D}}} \sigma^2(\mathcal S\mu).
%  = \sup_{\mu \in M_{\mathcal F},\ |\mu| \le \chi_{\mathbb{D}}} \sigma^2(\mathcal S\mu).
\end{equation*}
\end{theorem}
In view of Theorem \ref{thm:mcmullen61}, the first equality in Theorem \ref{fractal-approximation} is sufficient to deduce Theorem \ref{thm:sigmaanddimension}.
With a bit more work, the second equality %in Theorem \ref{fractal-approximation}
also gives the following consequence:

\begin{corollary} \label{cor:Juliaapprox}
For any $\varepsilon > 0$, there exists a family of polynomials
 \[
 z^d + t(a_{d-2} z^{d-2} + a_{d-3}z^{d-3} +  \dots + a_{0}), \qquad t \in (-\varepsilon_0, \varepsilon_0),
 \]
such that each Julia set $\mathcal J_t$ is a $k(t)$-quasicircle with $$\Hdim(\mathcal J_t) \ge 1 + (\Sigma^2 - \varepsilon)k(t)^2.$$
\end{corollary}

\subsection{Bounds on quadratic differentials}
To prove Theorem \ref{fractal-approximation}, we work with the integral average $\sigma^2_4$ rather than with $\sigma^2$. The reason for shifting the point
of view is due to the fact that the
pointwise estimates for
 \begin{equation}
\label{eq:bd2}
 v_\mu'''(z) = - \frac{6}{\pi} \int_{\mathbb{D}} \frac{\mu(w)}{(w-z)^4} dm(w)
 \end{equation}
are more useful than the pointwise estimates for $v'$, as we saw in Section \ref{section:background} when we invoked Hardy's identity.  
According to Lemma \ref{lemma:sigmas},
   \begin{equation}
\label{eq:mcm-formula2}
\sigma^2(v_\mu') = \sigma_4^2(v_\mu') =\ \frac{8}{3} \, \limsup_{R \to 1^+} \fint_{A(R,2)}  \biggl |\frac{v'''_\mu}{\rho_*^2} (z) \biggr |^2 \, \rho_*(z)dm
\end{equation}
where $\rho_*(z)=2/(|z|^2-1)$ is the density of the hyperbolic metric on $\mathbb{D}^*$
and $\fint f(z) \, \rho_*(z)dm$ denotes  the integral average with respect to the measure $\rho_*(z)dm$.
%(The expression in (\ref{eq:mcm-formula2}) is the analogue of the C\'esaro integral average (\ref{eq:mcm-cesaro}) for Bloch functions defined on the exterior unit disk. 
(Note that we are not taking the average with respect to the hyperbolic area $\rho_*^2(z)dm$.)

We will need two estimates for $v'''/\rho_*^2$.
To state these estimates, we introduce some notation.
For a set $E \subset \mathbb{C}$, let $E^*$ denote its reflection in the unit circle.
The hyperbolic distance between $z_1,z_2 \in \mathbb{D}^*$ is denoted by $d_{\mathbb{D}^*}(z_1,z_2)$.
The following lemma is based on ideas from \cite[Section 2]{mcm-cusps} and appears explicitly in \cite[Section 2]{ivrii-phd}:
\begin{lemma} 
\label{qbounds}
Suppose $\mu$ is a measurable Beltrami coefficient with $|\mu| \le \chi_{\mathbb{D}}$ and $v'''$ is given by
(\ref{eq:bd2}). Then,
 \begin{enumerate}
\item[$(a)$] $|v'''/\rho_*^2| \le 3/2$ for $z \in \mathbb{D}^*$.
\item[$(b)$] If $d_{\mathbb{D}^*}(z, \supp(\mu)^*) \ge L$, then $|(v'''/\rho_*^2)(z)| \leqslant C e^{-L}$, for some constant $C>0$.
\end{enumerate}
\end{lemma}

\begin{proof}
A simple computation shows that if $\gamma$ is a M\"obius transformation, then
\begin{equation}
\frac{\gamma'(z_1) \gamma'(z_2)}{(\gamma(z_1)-\gamma(z_2))^2} = \frac{1}{(z_1-z_2)^2}, \qquad \text{for }z_1 \ne z_2 \in \mathbb{C}.
\end{equation}
The above identity and a change of variables shows that 
\begin{equation}
\label{eq:mobiusinvariance}
v_\mu'''(\gamma(z)) \cdot \gamma'(z)^2  = v_{\gamma^*\mu}'''(z),
\end{equation}
analogous to the transformation rule of a quadratic differential.

In view of the M\"obius invariance, it suffices to prove the assertions of the lemma at infinity. 
From (\ref{eq:bd2}), one has 
\begin{equation*}
\label{eq:v3infinity}
\lim_{z \to \infty} \, \biggl |\frac{v_\mu'''}{\rho_*^2}(z) \biggr |=  \frac{3}{2\pi} \left| \int_{\mathbb D} \mu(w) dm(w) \right|,
\end{equation*}
which gives $(a)$. 
For $(b)$, recall that $d_{\mathbb{D}^*}(\infty,z) = -\log(|z|-1)+ \mathcal O(1)$ for $|z| < 2$. Then,
 \[
\lim_{z \to \infty} \, \biggl |\frac{v_\mu'''}{\rho_*^2}(z) \biggr | \le \frac{3}{2\pi} \int_{\{1- Ce^{-L} < |w| < 1\}} dm(w) = \mathcal O(e^{-L})
 \]
 as desired.
\end{proof}

\begin{remark}
Loosely speaking, part $(b)$ of Lemma \ref{qbounds} says that to determine the value of $v_\mu'''/\rho^2_*$ at a point $z \in \mathbb{D}^*$,
one needs to know the values of $\mu$ in a neighbourhood of $z^*$. More precisely, for any $\varepsilon > 0$, one may choose $L > 0$ 
sufficiently large to ensure that the contribution of the values of $\mu$ outside
$\{w : d_{\mathbb{D}}(z^*, w) < L\}$ to $(v_\mu'''/\rho^2_*)(z)$ is less than $\varepsilon$.
%(Note that the operator $\mu \to v_\mu'''/\rho^2_*$ is linear.)
In particular, if $\mu_1$ and $\mu_2$ are two Beltrami coefficients, supported on the unit disk and bounded by 1 that agree on
$\{w : d_{\mathbb{D}}(z^*, w) < L\}$, then $\bigl | (v_{\mu_1}'''/\rho^2_*)(z) - (v_{\mu_2}'''/\rho^2_*)(z)  \bigr | < 2\varepsilon.$
This simple {\em localisation principle} will serve as the foundation for the arguments in this section.
\end{remark}

\begin{lemma}
\label{qd-almostinvariant2}
 Given an $\varepsilon > 0$, there exists an $1 < R(\varepsilon) < \infty$, so that if
 $1<|z|^d < R(\varepsilon)$, then
\begin{equation}
\biggl |z^2 \frac{v_{(z^d)^*\mu}'''}{\rho_*^2}(z) - z^{2d}\frac{v_{\mu}'''}{\rho_*^2}(z^d) \biggr | < \varepsilon.
\end{equation}
Note that $R(\varepsilon)$ is independent of the degree $d \ge 2$.
\end{lemma}

\begin{proof}
Differentiating \eqref{eq:pullback1} three times yields 
\[\left| d^2 z^{2d-2}v_{\mu}'''(z^d)-v_{(z^d)^*\mu}'''(z) \right| \leqslant 2 d^2 |z|^{-2} \omega(z^d),
\]
where $\omega(z)/\rho_*^2(z)\to 0$ as $|z|\to 1^+$. 
 The lemma follows in view of the convergence
 $(1/d) \cdot (\rho_*(z)/\rho_*(z^d)) \to 1$ as $|z|^d\to 1^+,$
  which is uniform over $d\ge 2$.

Alternatively, one can use a  version of Koebe's distortion theorem for maps which preserve the unit circle,  see \cite[Section 2]{ivrii-phd}.
\end{proof}

\subsection{Periodising Beltrami coefficents}
We now prove the first equality in Theorem \ref{fractal-approximation} which says that 
$\Sigma^2 =  \sup_{\mu \in M_{\ei}, \ |\mu| \le \chi_{\mathbb{D}}}  \sigma^2(\mathcal S\mu)$. 
In view of Lemma \ref{lemma:sigmas}, given a Beltrami coefficient $\mu$ with $|\mu|  \le \chi_{\mathbb{D}}$ and $\varepsilon>0$, it suffices to construct an eventually-invariant Beltrami coefficient $\mu_d$ which satisfies
\begin{equation}
\label{eq:fractal-approximation1}
|\mu_d|   \le \chi_{\mathbb{D}} \qquad \text{and} \qquad
 \sigma^2_4(v'_{\mu_d}) \ge \sigma^2_4(v'_\mu) - \varepsilon.
\end{equation}

\begin{proof}[Proof of Theorem \ref{fractal-approximation}, first equality]
Given an integer $d \ge 2$, we construct a Beltrami coefficient $\mu_d \in M_{\ei}(d)$. 
We then show that $\mu_d$ satisfies (\ref {eq:fractal-approximation1})
for  $d$ sufficiently large.

\medskip

{\em Step 1.} 
Using the definition of asymptotic variance (\ref{eq:mcm-formula2}), we select an  
 annulus $$A_0^* = A(R_1, R_0) \subset \mathbb{D}^*, \quad R_1 = R_0^{1/d}, \quad R_0 \approx 1,$$
  for which
\begin{equation*}
\sigma^2_4(v'_\mu) - \varepsilon/3 \le
\frac{8}{3} \, \fint_{A_0^*} \biggl |\frac{v'''_\mu}{\rho_*^2}(z) \biggr |^2 \rho_*(z) dm.
\end{equation*}
Let $A_0 = A(r_0, r_1)$ be the reflection of $A_0^*$ in the unit circle. 
We take $\mu_d = \mu$ on $A_0$ and then extend $\mu_d$
to $\{z : r_1 < |z| < 1\}$ by $z^d$-invariance. On $|z| < r_0$, we set $\mu_d = 0$.

\medskip

\noindent {\em Step 2.} 
The estimate (\ref{eq:fractal-approximation1}) relies on an isoperimetric feature of the measure $\rho_*(z)dm$, which we now describe.
It is easy to see that the $\rho_*(z)dm$-area of an annulus $$A(S_1, S_2), \quad 1 < S_1 < S_2 < \infty,$$
  is
$2\pi$ times the hyperbolic distance between its boundary components. In particular, the $\rho_*(z)dm$-area of $A_0^*$ is roughly $2\pi \log d$. By contrast, for a fixed $L > 0$, the  $\rho_*(z)dm$-area of its ``periphery''
$$
\partial_L A_0^* := \{z \in A_0^*, \, d_{\mathbb{D}}(z, \partial A_0^*) < L \}
$$
is $4\pi L$ (provided that $\log d \ge 2L$). We conclude that the ratio of $\rho_*(z)dm$-areas of $\partial_L A_0^*$ and $A_0^*$ tends to 0 as
$d \to \infty$.  

\medskip

\noindent {\em Step 3.} 
By part $(b)$ of Lemma \ref{qbounds}, %we have for every $\varepsilon_0>0$
\begin{equation}
\label{eq:six-1}
\biggl |\frac{v'''_{\mu_d}}{\rho_*^2} (z) - \frac{v'''_{\mu}}{\rho_*^2} (z) \biggr | \le Ce^{-L},  \quad z \in A_0^* \setminus \partial_L  A_0^*,
\end{equation}
while
\begin{equation}
\label{eq:six-2}
\biggl |\frac{v'''_{\mu_d}}{\rho_*^2} (z) - \frac{v'''_{\mu}}{\rho_*^2} (z) \biggr | \le 3, \quad z \in \partial_L  A_0^*.
\end{equation}
Putting the above estimates  together gives (for large degree $d$)%\eqref{eq:six-1} and \eqref{eq:six-2}
\begin{equation}
\label{eq:annulus-estimate}
\left |\,\frac{8}{3} \, \fint_{A_0^*} \biggl | \frac{v'''_{\mu_d}}{\rho_*^2} (z)\biggr|^2 \rho_*(z)dm  - \frac{8}{3} \, \fint_{A_0^*} \biggl|\frac{v'''_\mu}{\rho_*^2}(z) \biggr |^2 \rho_*(z)dm \, \right| < \varepsilon/3.
\end{equation}

\smallskip

\noindent {\em Step 4.} 
Set $R_k := R_0^{1/d^k}$ and $A_k^* =A(R_{k+1},R_k)$. By Lemma \ref{qd-almostinvariant2}, 
$$
\left |\, \frac{8}{3}\,\fint_{A_k^*} \biggl |\frac{v'''_{\mu_d}}{\rho_*^2}(z) \biggr |^2 \rho_*(z) dm
\ - \ \frac{8}{3} \,
\fint_{A_0^*} \biggl |\frac{v'''_{\mu_d}}{\rho_*^2}(z) \biggr |^2 \rho_*(z) dm \, \right | < \varepsilon/3, $$
which implies that $\sigma^2_4(v'_{\mu_d}) \ge \sigma^2_4(v'_{\mu})-\varepsilon$ as desired.
\end{proof}

\begin{remark}
(i) The isoperimetric property used above does not hold with respect to the hyperbolic area $\rho_*^2(z) dm$. In fact, as we explain in Section \ref{se:Fuchsian},  periodisation fails in the Fuchsian case.

(ii) Refining the above argument shows that one can take $\varepsilon = C/\log d$ in  (\ref {eq:fractal-approximation1}), but we will not need this more quantitative estimate.
\end{remark}

\subsection{Polynomial perturbations}
\label{subsection:pp} 
To show the second equality in Theorem \ref{fractal-approximation}, we need a description of vector fields 
 which arise from polynomial perturbations of $z \to z^d, \, d \ge 2$. 
\begin{lemma} \cite[Section 5]{mcmullen}
\label{lemma:polynomials}
Consider the family of polynomials
\begin{equation} \label{sarja}
P_t(z) = z^d + t\, Q(z), \qquad \deg Q \leq d-2, % (a_d + a_{d-1}z + a_2z^2 + \dots + a_{2}z^{d-2}), 
 \quad |t| < \varepsilon_0.
\end{equation}
Let $\varphi_t: \mathbb{D}^* = A_{P_0}(\infty) \to A_{P_t}(\infty)$ denote the conjugacy map
 and $v = \frac{d\varphi_t}{dt}\bigl |_{t = 0}$ be the associated vector field as before.
Then,
\begin{equation} \label{sarja2}
v(z) = \sum_{k=0}^\infty v_k(z) =  \frac{z}{d}\sum_{k \ge 0} \frac{Q( z^{d^{k}})}{ d^{k} z^{d^{k+1}} }, \qquad z \in \DD^*.
\end{equation}
\end{lemma}
Let $\Vpp(d)$ be the collection of holomorphic vector fields of the form \eqref{sarja2}, with $\deg Q \le d-2$.
From this description, it is clear that each $\Vpp(d), \, d \ge 2$ is a vector space, but the union $\Vpp = \bigcup_{d \ge 2} \Vpp(d)$ is not.
Observe that two consecutive terms in \eqref{sarja2} satisfy the ``periodicity'' relation
\begin{equation}
\label{eq:periodicity}
v_{k+1}(z)=\frac{1}{dz^{d-1}} v_k(z^d),
\end{equation}
which is of the form \eqref{eq:pullback1} provided that $\mathcal C \mu(0)=0$. 

Similarly, we define $M_{\pp} = 
\bigcup_{d \ge 2}M_{\pp}(d)$ as the class of Beltrami coefficients that give rise to polynomial perturbations. More precisely, $M_{\pp}(d)$ consists
of  eventually-invariant Beltrami coefficients $\mu \in M_{\ei}(d)$
 for which $v_\mu=\mathcal C \mu \in \Vpp(d)$. 
 %Clearly, the condition $\mathcal C\mu(0) = 0$ is satisfied for $\mu \in M_{\pp}$.

\subsection{A truncation lemma}
In order to approximate infinite series by finite sums, we need some kind of a truncation procedure. To this end, we show the following lemma:

\begin{lemma}
\label{new-truncation-lemma}
Suppose $\mu$ is a Beltrami coefficient satisfying $\|\mu\|_\infty \le 1$ and $\supp \mu \subset A(\rho_0, \rho_1)$, with $0 < \rho_0 < \rho_1 < 1$.
Given a slightly larger annulus $A(\rho_0, r_1)$ % $0 < r_0 < \rho_0 < \rho_1 < r_1 < 1$,
and an $\varepsilon > 0$, there exists a Beltrami coefficient $\tilde \mu$ satisfying
 
{\em (i)} $\supp \tilde \mu \subset A(\rho_0, r_1),$

{\em (ii)} $ \|\tilde \mu - \mu\|_\infty < \varepsilon,$

{\em (iii)} $
v_{\tilde \mu}(0) = v_\mu(0),
$

{\em (iv)} $
v_{\tilde \mu} \text{ is a polynomial in }z^{-1}.
$
\end{lemma}

\begin{proof}
 From
\begin{equation*}
v_{\mu}(z) =  \frac{1}{\pi z } \int_{\mathbb{D}} \mu(w) \Bigl (1 + w/z + w^2/z^2 + \dots \Bigr ) dm(w),
\end{equation*}
it follows that
\begin{equation*}
v_{\mu} = \sum_{j=1}^\infty b_j z^{-j}, \qquad b_j =  \frac{1}{\pi } \int_{\mathbb{D}} \mu(w) w^{j-1} dm(w).
\end{equation*}
Since $\mu$ is supported on $A(\rho_0, \rho_1)$, the coefficients $b_j$ decay exponentially, more precisely, $|b_j| \le  \frac{2}{j+1} (\rho_1^{j+1}-\rho_0^{j+1})$.
As $\rho_1/r_1 < 1$, for $N$ sufficiently large, we have $$\sum_{j \ge N+1} \frac{|b_j|}{\frac{2}{j+1} \bigl ( r_1^{j+1}-\rho_0^{j+1} \bigr )}
\le
\sum_{j \ge N+1} \frac{\rho_1^{j+1}-\rho_0^{j+1}}{r_1^{j+1}-\rho_0^{j+1}}
\le
\sum_{j \ge N+1} \frac{\rho_1^{j+1}}{r_1^{j+1}}
 \le \varepsilon.$$
Using Lemma \ref{lemma:basic-coefficient}, is easy to see that
$$
\tilde{\mu} = \mu - \sum_{j \ge N+1} \frac{b_j}{\frac{2}{j+1} \bigl ( r_1^{j+1}-\rho_0^{j+1} \bigr )}  \cdot \biggl ( \frac{\overline{z}}{|z|} \biggr )^{j-1} \cdot \chi_{A(\rho_0,r_1)}(z)
$$
satisfies the desired properties.
\end{proof}

\subsection{Periodising quadratic differentials} With these preliminaries, we can complete the proof of Theorem \ref{fractal-approximation}.

\begin{proof}[Proof of Theorem \ref{fractal-approximation}, second equality]
From the proof of the first part of the theorem, we may assume that $\mu$ is an eventually-invariant Beltrami coefficient of the form
$
\mu = \mu_0  + \mu_1 + \dots$
where
$$
 \mu_k = (z^{d^k})^*\mu_0, \quad \supp \mu_k \subset A_k = A(r_k, r_{k+1}), \quad  r_k = r_0^{1/d^{k}},
\quad 0<r_0<1.
$$
Furthermore, it will be convenient to assume that $\mu_0$ itself arises as a pullback under $z \to z^d$, which by Remark \ref{rem:cauchy0}
implies that $v_{\mu_k}(0)=0$ for all $k \ge 0$. This could be achieved by considering $(z^d)^*\mu$ instead of $\mu$ and renaming $r_1$ by $r_0$.  

\medskip

\noindent {\em Step 1.} 
We now show that we may additionally assume that $v_{\mu_0}$ is a polynomial in $z^{-1}$. For this purpose, we first replace
 $\mu_0$ by $\mu_0 \cdot \chi_{A(r_0, \rho_1)}$ so that $\supp \mu_0$ is contained in a slightly smaller annulus $A(r_0, \rho_1) \subset A(r_0, r_1)$.
 We then apply Lemma \ref{new-truncation-lemma} with $\mu = \mu_0$ to
 obtain a Beltrami coefficient $\tilde \mu_0$ supported on $A(r_0, r_1)$ with the desired property.
Finally, we replace $\tilde \mu_0$ by $\tilde \mu_0/(1+\varepsilon)$ to ensure that $\|\tilde \mu_0\|_\infty \le 1$.
Since all three operations
have little effect on the integral
$$
\fint_{A(1/r_1, 1/r_0)} \biggl |\frac{v'''_\mu}{\rho_*^2}(z) \biggr |^2 \rho_*(z) dm,
$$
we see that
 $\sigma_4^2(v'_{\tilde \mu}) \approx  \sigma_4^2(v'_\mu)$  where $\tilde \mu := \sum_{k \ge 0} \tilde \mu_k = \sum_{k \ge 0} (z^{d^k})^* \tilde \mu_0$.

\smallskip

\noindent {\em Step 2.} In view of Lemma \ref{pullback-coefficients}, the sequence $v_k  = v_{\tilde \mu_k}$ % have disjoint coefficients and 
 satisfies the degree $d$ periodicity relation \eqref{eq:periodicity}. However,
 we cannot guarantee that 
 $v = \sum_{k=0}^\infty v_k \in \Vpp(d)$
since the base polynomial $v_0$ may have degree greater than $d-1$ in $z^{-1}$.
Let $m$ be the smallest integer so that $\deg_{z^{-1}} v_0 \le d^m - 1$, and take $M>m$. Consider then the Beltrami coefficient
$\hat \mu = \sum \hat \mu_k$ where
$$
\hat \mu_0 =  \tilde \mu_0 + \tilde \mu_1 + \dots + \tilde \mu_{M-m}
 \quad \text{and} \quad \hat \mu_k =  (z^{k d^M})^*{\hat \mu_0}.
$$
Similarly, define
$$\hat v_0  = \mathcal C\hat \mu_0 = v_0 + v_1 + \dots + v_{M-m},$$
$$ \hat v_k = \mathcal C\hat \mu_k \quad \text{and} \quad \hat v = \sum \hat v_k.$$
By construction, $\hat v$ is the periodisation of $\hat v_0$
under the relation \eqref{eq:periodicity}, with $d^M$ in place of $d$. Since $\deg_{z^{-1}} v_{M-m} \leqslant d^{M-m} \deg_{z^{-1}} v_0 + (d^{M-m}-1)$, we have $\deg_{z^{-1}} \hat v_0 \le d^M - 1$ which ensures that $\hat v \in \Vpp(d^M)$.
Explicitly, $\hat v$ is the vector field associated to the polynomial perturbation %, c.f. \eqref{sarja}:
$$
P_t(z) = z^{d^M} + t \cdot d^M z^{d^M-1} \, \hat v_0(z), \qquad |t| < \varepsilon_0.
$$

By taking $M >\!\!\!> m$, the fraction of the ``unused'' shells (i.e.~those corresponding to indices
$M-m+1, \dots, M-1$) can be made arbitrarily small.
Using Lemma \ref{qbounds} $(b)$ like in Step 3 in the proof of the first equality in Theorem \ref{fractal-approximation} shows that  $\sigma^2_4(v'_{\hat \mu}) \approx \sigma^2_4(v'_{\tilde \mu})$ as desired.
 \end{proof}

\begin{proof}[Proof of Corollary \ref{cor:Juliaapprox}]
%In order to prove the corollary, we repeat some of the arguments in Theorem \ref{improvement}.
By the second equality in Theorem \ref{fractal-approximation}, for $\varepsilon>0$, one can find a Beltrami coefficient $\mu \in M_{\pp}$ with 
$|\mu| \le \chi_{\mathbb{D}}$ for which $\sigma^2(\mathcal{S}\mu) > \Sigma^2-\varepsilon$.
By the definition of $M_{\pp}$, the associated vector field lies in $\Vpp$. By Lemma \ref{lemma:polynomials}, there exists a family of polynomials 
\[ P_t(z) = z^d + t Q(z), \quad \deg Q \le d-2, \quad |t|<\varepsilon_0,
\]
with
\begin{equation}
\label{eq:v=C}
\frac{d}{dt}\biggl |_{t = 0} \varphi_t(z) \ = \  \mathcal{C}\mu(z) \ = \ \frac{z}{d}\sum_{k \ge 0} \frac{Q( z^{d^{k}})}{ d^{k} z^{d^{k+1}} }, \qquad z \in \mathbb{D}^*,
\end{equation}
where $\varphi_t: \mathbb{D}^* = A_{P_0}(\infty) \to A_{P_t}(\infty)$ are conformal conjugacies.
We are now in a position to repeat the argument in the proof of Theorem \ref{improvement}. 
Indeed, by the $\lambda$-lemma, the conformal maps $\varphi_t$ admit \emph{some} quasiconformal extensions $H_t \colon \mathbb{C} \to \mathbb{C}$.
Using \eqref{eq:v=C}, for $|t|<\varepsilon_0$, we can correct the extensions $H_t$ by %$H_t \circ N_t^{-1}$ as in 
pre-composing them with Teichm\"uller-trivial deformations $N_t^{-1}$ like in 
\eqref{eq:correctedextension}, so that
\[ \mu_{H_t \circ N_t^{-1}} = t \mu+\mathcal{O}(t^2).
\]
Therefore, the Julia sets $\mathcal J_t=\mathcal J(P_t)$ are $k(t)$-quasicircles with
\[ k(t)=\frac{|t|}{2}+ \mathcal{O}(|t|^2), \quad \mbox{as $t \to 0$}.
\]
On the other hand, their Hausdorff dimensions satisfy 
\[ \Hdim \mathcal J_t=1+ \sigma^2(v_\mu') \frac{|t|^2}{4}+\mathcal{O}(|t|^3).
\]
Since $\sigma^2(v_\mu')=\sigma^2(\mathcal S\mu)>\Sigma^2-\varepsilon$, letting $t \to 0$ proves the claim.
\end{proof}

\section{Fuchsian groups}
\label{se:Fuchsian}

One may ask whether 
\begin{equation}
\label{eq:fuchsian-question}
\Sigma^2 \, \stackrel{?}{=} \, \sup_{\mu \in M_{\fuchs},\ |\mu| \le \chi_{\mathbb{D}}} \sigma^2(\mathcal S\mu),
\end{equation} 
 for the class $M_{\text{F}} = \bigcup_\Gamma M_{\Gamma}(\mathbb{D})$  of Beltrami coefficients  that are invariant under some co-compact Fuchsian group $\Gamma$,
i.e.~$\gamma^*\mu = \mu$ for all $\gamma \in \Gamma$.
It is tempting to take a Beltrami coefficient $\mu$ on the unit disk and  periodise it
with respect to a Fuchsian group $\Gamma$ of high genus, i.e.~to form a $\Gamma$-invariant Beltrami coefficient
$\mu_F$ which coincides with $\mu$ on a fundamental domain $F \subset \mathbb{D}$. However, one cannot guarantee that $\sigma^2(v'_{\mu_F}) \approx \sigma^2(v'_\mu)$.

The reason for this is that the hyperbolic area of $F$ is comparable to the hyperbolic area of its ``periphery''
$$
\partial_1 F := \{z \in F, \, d_{\mathbb{D}}(z, \partial F) < 1 \}.
$$
Unlike our considerations in complex dynamics (with the maps $z \to z^d$), in the Fuchsian case, the periphery is significant:  
Indeed, if $\pi: \mathbb{D} \to \mathbb{D}/\Gamma$ denotes the universal covering map, it is well-known that as $r \to 1$, the curves $\pi(\{z : |z|=r\})$ become equidistributed with respect to the hyperbolic metric on $\mathbb{D}/\Gamma$. Therefore, for $r$ close to 1, the curves $\pi(\{z : |z|=r\})$ spend a definite amount of time in $\partial_1 F \subset F \cong
 \mathbb{D}/\Gamma$, which allows the asymptotic variance to go down after periodisation. Indeed,
fractal approximation fails in the Fuchsian case as evidenced by Theorem \ref{fuchsian-case}.

Theorem \ref{fuchsian-case} is a simple consequence of the comparison
\begin{equation}
\label{eq:tangent-estimate}
\frac{\|\mu\|_{\WP}^2}{\Area(X)} \le \|\mu\|_T^2
\end{equation}
 between the Teichm\"uller and Weil-Petersson metrics on the
Teichm\"uller space $\mathcal T_g$ of compact Riemann surfaces of genus $g \ge 2$, for instance, see \cite[Proposition 2.4]{mcm-kahler}. Here,
the area is taken with respect to the hyperbolic metric on $X$.
For the convenience of the reader, we
recall the definitions and sketch the rather simple arguments.

For $X \in \mathcal T_g$, the cotangent space $T_X^* \mathcal T_g$ is canonically identified with the space $Q(X)$
of holomorphic quadratic differentials $q(z)dz^2$ on $X$. On the cotangent space, the Teichm\"uller and Weil-Petersson norms 
are given by
$$
\|q\|_T := \int_X |q|, \qquad \|q\|^2_{\WP} := \int_X |q|^2 \rho^{-2},
$$
where $\rho$ denotes the density of the hyperbolic metric.
Dualising shows that the tangent space $T_X \mathcal T_g \cong M(X)/Q(X)^\perp$ is naturally identified with the quotient space of Beltrami coefficients
$\mu \in L^\infty(X)$ modulo ones orthogonal to $Q(X)$ with respect to the pairing
$\langle \mu, q \rangle = \int_X \mu q$. The Teichm\"uller and Weil-Petersson metrics on the tangent space may be obtained by dualising the above
definitions:
$$
\|\mu\|_T := \sup_{\|q\|_T =1 } \, \biggl | \int_X \mu q \biggr |, \qquad \|\mu\|_{\WP} := \sup_{\|q\|_{\WP} =1 } \, \biggl | \int_X \mu q \biggr |.
$$
Duality considerations also show that (\ref{eq:tangent-estimate}) is equivalent to the estimate
\begin{equation}
\label{eq:cotangent-estimate}
\Area(X) \, \|q\|_{\WP}^2 \ge \|q\|_T^2,
\end{equation}
which is immediate from the Cauchy-Schwarz inequality.

To show that $\Sigma^2_{\fuchs} \le 2/3$, it suffices to describe the standard geometric interpretations for the dual norms.
For the Teichm\"uller norm, one has
\begin{equation}
\label{eq:teich-identification}
\|\mu\|_T \le \|\mu\|_\infty.
\end{equation}
In fact, $\|\mu\|_T = \inf_{\nu \sim \mu} \|\nu\|_\infty$, where the infimum is taken over all $\nu$ infinitesimally equivalent to $\mu$
(with $\mathcal C\nu = \mathcal C\mu$ on $\mathbb{D}^*$)
and the minimum is achieved for a unique Teichm\"uller coefficient of the form $k \cdot \frac{\overline{q}}{|q|}$ with $q \in Q(X)$.

On the other hand, the Weil-Petersson norm may be computed by
\begin{eqnarray}
\frac{\|\mu\|_{\WP}^2}{\Area(X)}
& = &  \frac{1}{\Area(X)} \int_X  |2v''' /\rho^2|^2 \, \rho^{2} dxdy, \nonumber \\
& = & \lim_{R \to 1^+} \frac{1}{2\pi} \int_{|z|=R} |2v'''/\rho^2|^2 d\theta, \nonumber \\
\label{eq:wp-identification}
& = & \frac{3}{2} \sigma^2(v'_\mu).
\end{eqnarray}
The first equality follows from the fact that $\rho^{-2} \cdot (1/\overline{z})^* \bigl [ -2v'''(z) dz^2 \bigr ]$ is the harmonic representative of the Teichm\"uller class of $[\mu]$, e.g.~see \cite[Chapter 7]{IT}, the second
from equidistribution, and the third from Lemma \ref{lemma:sigmas}. 
Substituting (\ref {eq:teich-identification}) and (\ref{eq:wp-identification}) into (\ref{eq:tangent-estimate}) gives $\Sigma^2_{\fuchs} \le 2/3$.

To complete the proof of Theorem \ref{fuchsian-case}, it suffices to show that there is a definitive defect in the Cauchy-Schwarz inequality
(\ref{eq:cotangent-estimate}), i.e.~that there exists a constant $\varepsilon_0 > 0$ for which
\begin{equation}
\label{eq:cotangent-estimate2}
(1 - \varepsilon_0) \cdot \Area(X) \, \|q\|_{\WP}^2 \ge  \|q\|_T^2,
\end{equation}
independent of $g \ge 2$, $X \in \mathcal T_g$ and $q \in Q(X)$.
For this purpose, we use the following general fact: for non-zero vectors $\mathbf{x}$ and $\mathbf{y}$ in a Hilbert space $\mathcal H$,
\begin{equation}
\label{eq:hilbert-space-lemma}
\biggl \| \frac{\mathbf{x}}{\| \mathbf{x} \|} - \frac{\mathbf{y}}{\| \mathbf{y} \|} \biggr \| > \delta \quad \implies \quad
\bigl |\langle \mathbf{x}, \mathbf{y} \rangle \bigr | < (1 - \varepsilon) \, \| \mathbf{x} \| \, \| \mathbf{y} \|.
\end{equation}
In our setting, $\mathcal H = L^2(X, \rho^2)$, $\mathbf{x} = |q|/\rho^2$ and $\mathbf{y} = 1$. To make use of the above motif, we need to show that:
\begin{lemma}
There exists a positive constant $\delta > 0$ so that
\begin{equation}
\label{eq:positivity}
\int_X \, \biggl | \frac{|q|}{\rho^2} - 1 \biggr |^2 \, \rho^2 dxdy \, > \, \delta^2 \int_X \, \rho^2 dxdy.
\end{equation}
independent of the Riemann surface $X$ and $q \in Q(X)$.
 \end{lemma}
 \begin{proof}
 To prove the lemma, we first show that there exists a $\delta > 0$ such that
 \begin{equation}
\label{eq:positivity2}
\int_{B(0,1/2)} \, \biggl | \frac{|h|}{\rho^2} - 1 \biggr |^2 \, \rho^2 dxdy \, > \, \delta^2 \int_{B(0,1/2)} \, \rho^2 dxdy,
\ \ \text{with }\rho = \rho_{\mathbb{D}},
\end{equation}
for any function $h$ holomorphic in a neighbourhood of $\overline{B(0,1/2)}$.
This follows from a simple compactness argument, since any potential minimiser $h$ has bounded $L^2$ norm on $B(0,1/2)$, and  $\rho^2$ is not the absolute value of any holomorphic function. One way to see this
is to observe that $\Delta \log|h| = 0$ but $\Delta \log \rho^2 = 2 \rho^2$ (the Poincar\'e metric has constant curvature $-1$).

Replacing $h(z)$ by $h(\gamma(z))\gamma'(z)^2$, $\gamma \in \Aut(\mathbb{D})$ we see that the estimate (\ref{eq:positivity2}) holds on any ball $B_{\hyp}(z, R) \subset \mathbb{D}$ of hyperbolic radius $R= d_{\mathbb{D}}(0,1/2)$.
Choose a covering map $\pi: \mathbb{D} \to X$ and lift $\tilde q = \pi^*q$ to the disk.
In view of equidistribution, to obtain (\ref{eq:positivity}), we take $h = \tilde q$ and average
 (\ref{eq:positivity2}) over balls $B_{\hyp}(z,R)$ whose centers lie on $\{z : |z| = r\}$ with $r \approx 1$.
 Taking $r \to 1$ completes the proof.
 \end{proof}
 
\begin{remark} Note that Theorem \ref{fuchsian-case} does not show that the limit sets $w^{t\mu}(\mathbb{S}^1)$ cannot be expressed as
quasicircles of dimensions greater than $1+(2/3)k(t)^2$, for $t$ small, only that representations using {\em invariant} Beltrami coefficients are inadequate.

In fact, Kra's $\theta$ conjecture (proved by McMullen in \cite{mcm-amen}, see \cite{stergios-theta} for a simple proof) implies that given an invariant Beltrami coefficient
$\mu \in M_\Gamma(\mathbb{D})$, there necessarily exists a (non-invariant) Beltrami coefficient $\nu \in M(\mathbb{D})$
infinitesimally equivalent to $\mu$ with $\|\nu\|_\infty < \|\mu\|_\infty$.
\end{remark}

\begin{comment}
\begin{remark}
The methods of Section \ref{sec:fractal} show that nothing is gained by allowing Riemann surfaces with punctures, that is, given a Beltrami coefficient
$\mu$ on the unit disk, invariant by a Fuchsian group of co-finite area, there exists a Beltrami coefficient $\nu$, invariant under a co-compact Fuchsian group so that $\sigma^2(\mathcal S\nu) > \sigma^2(\mathcal S\mu) -\varepsilon.$
\end{remark}
\end{comment}

\section{Dynamical analogue of asymptotic variance}
\label{sec:dynamics}

In this section, we discuss the notion of asymptotic variance of a H\"older continuous potential from thermodynamic formalism.
 Using a global analogue of McMullen's coboundary equation \cite[Theorem 4.5]{mcmullen},
we relate it to the notion of asymptotic variance of a Bloch function considered earlier. As an application,
we obtain estimates for the integral means spectrum of univalent functions.

For concreteness, we work with a certain class of fractals arising from quasiconformal deformations of Blaschke products and leave the general case to the reader, see Remark \ref{rem:general-case}.
\subsection{Thermodynamic formalism} 
Let $$B(z) = z \prod_{i=1}^{d-1} \frac{z-a_i}{1-\overline{a_i}z}, \qquad a_i \in \mathbb{D},$$ be a  finite Blaschke product, which we think of
as a map from the unit circle to itself.
Let $m$ denote the Lebesgue measure on the unit circle, normalised to have total mass 1. It is well-known that the Lebesgue measure
is invariant under $B$, that is, $m(E) = m(B^{-1}(E))$.

For a  H\"older continuous potential  $\phi \in C^\alpha({\mathbb S}^1)$
of mean zero, i.e.~with $\int \phi dm = 0$, the ``dynamical'' asymptotic variance is given by
\begin{equation}
\var(\phi) := \lim_{n \to \infty}  \frac{1}{n} \int_{{\mathbb S}^1} |S_n \phi(z)|^2 dm,
\end{equation}
where $S_n \phi(z) = \sum_{k=0}^{n-1} \phi(B^{\circ k}(z))$. More generally, for  $\phi, \psi \in C^\alpha({\mathbb S}^1)$ with
$\int \phi dm = \int \psi dm = 0$, one may consider the {\em covariance}
\begin{equation}
\var(\phi, \psi) := \lim_{n \to \infty}  \frac{1}{n} \int_{{\mathbb S}^1} S_n \phi(z) \overline{S_n \psi(z)} dm.
\end{equation} 
To show that $\var(\phi)$ and $\var(\phi, \psi)$ are well-defined, one may use the exponential decay of correlations,
\cite[Theorem 4.4.9]{PU}
or \cite[Proposition 2.4]{PP},
\begin{equation}
\label{eq:decay-corr}
\int_{{\mathbb S}^1} \phi(B^{\circ j}(z)) \overline{\psi(B^{\circ k}(z))} dm \le K \theta^{k-j} \|\phi\|_{C^\alpha} \|\psi\|_{L^1}, \quad j \le k,
\end{equation}
for some  $0 < \theta(\alpha, B) < 1$. %The key point here is that the estimate (\ref{eq:decay-corr}) is uniform in terms of the H\"older norm of $\phi$.
 In particular, the functions
 \begin{equation}
 \var_n(\phi) := \frac{1}{n} \int_{{\mathbb S}^1} |S_n \phi(z)|^2 dm
 \end{equation} converge uniformly to $\var(\phi)$.

Following \cite{mcmullen}, we say that $h \in C^{\alpha}({\mathbb S}^1)$ is a {\em virtual coboundary} of  $g \in C(A(1,R))$, $R > 1$, if
the difference
$
g(z) - g(B(z))
$
extends to a continuous function on the unit circle and the extension coincides with $h$. We will need the following fundamental result about virtual coboundaries:
\begin{theorem}
\label{virtual-cocycle-thm}
 \cite[Theorem 4.1] {mcmullen}
Suppose $h \in C^\alpha(\mathbb{S}^1)$, $0 < \alpha < 1$, of mean zero, can be expressed as a virtual coboundary $g(z) - g(B(z))$.
Then the limit in the definition of $\sigma^2(g)$ exists and
\begin{equation}
\frac{\var(h)}{\int \log|B'|dm} = \sigma^2(g).
\end{equation}
\end{theorem}
\begin{remark}
By itself, the fact that $h$ is a virtual coboundary does not guarantee that $h$ has mean 0. For instance, if $B(z) = z^2$, then the constant
function $h(z) = \log 2$ is a  virtual coboundary of $g(z) = \log\frac{1}{|z|-1}$. Using arguments similar to those discussed in \cite[Section 3]{mcmullen},
one can show the relation
\begin{equation}
\label{eq:meanofh}
\frac{\int_{\mathbb{S}^1} h(z) dm}{\int_{\mathbb{S}^1} \log|B'|dm} \, = \, \lim_{R \to 1^+} \frac{1}{2\pi|\log(R-1)|} \int_{|z|=R} g(z) \, |dz|.
\end{equation}
%However, if $g$ is harmonic (or holomorphic) on $\mathbb{S}^2 \setminus \mathbb{D}$, then the right hand side of (\ref{eq:meanofh}) is 0 by the mean-value theorem, in which case, $\int hdm=0$.
\end{remark}

\subsection{Conformal maps with fractal boundaries}
Suppose $\mu \in M_{B}(\mathbb{D})$ is an eventually-invariant Beltrami coefficient supported on the unit disk with
$\|\mu\|_\infty < 1$. 
By solving the Beltrami equation, $\mu$ gives rise to a conformal map $H(z) = w^\mu(z)$ of the exterior unit disk.
We may use the conjugacy $H$ to transfer the dynamics of $B$ to $H(\mathbb{S}^1)$, in which case the map $$F = H \circ B \circ H^{-1}$$ is a dynamical system on the image curve $H(\mathbb{S}^1)$.
Since $\mu$ is eventually-invariant, the map $F$ is holomorphic in a neighbourhood of $\overline{H(\mathbb{D}^*)}$.

Associated to the map $F$, we define the potential
\begin{eqnarray}
\label{eq:virtual1}
\psi(z) & = &  \log F'(H(z)) - \log B'(z), \qquad 1-\varepsilon < |z| < 1 + \varepsilon.\\
\label{eq:virtual}
 & = & \log H'(B(z)) - \log H'(z), \qquad z \in \mathbb{D}^*.
\end{eqnarray}
Observe that the two definitions are complementary to each other: the first definition makes sense near the unit circle, 
while the second definition is good in the exterior unit disk but does not work on the unit circle. To see the equivalence of the two definitions,
it suffices to differentiate the conjugacy relation $F \circ H = H \circ B$.

The first definition implies that $\psi$ is H\"older continuous on the unit circle. 
The virtual coboundary condition (\ref{eq:virtual}) together with the fact that $B(\infty) = \infty$ guarantee that %the potential $\psi$ has mean 0:
$$\int_{\mathbb{S}^1} \psi(z) dm = 0$$ by the mean-value theorem.
Applying Theorem \ref{virtual-cocycle-thm} gives
\begin{equation}
\label{eq:missinglink}
\frac{\var(\psi)}{\int \log|B'|dm} = \sigma^2(\log H').
\end{equation}
As was explained in the introduction, this identity completes the proof of Theorem \ref{dynamical-connections}.

\begin{remark}
\label{rem:general-case}
The argument presented here (with some modifications) also applies to a wider class of fractals known as {\em Jordan repellers} $(J, F)$ which are defined
by the following conditions:

(i) The set  $J$ is a Jordan curve,
presented as a union of closed arcs
 $J = J_1 \cup J_2 \cup \dots \cup J_n,$ with pairwise disjoint interiors.

(ii) For each $i = 1, 2, \dots, n$, there exists a univalent function $F_i: U_i \to \mathbb{C},$ defined on a neighbourhood $U_i \supset J_i$, such that $F_i$ maps $J_i$ bijectively onto the union of several arcs, i.e.~$$F_i(J_i) = \bigcup_{j \in \mathcal A_i} J_j,$$ 

(iii) Additionally, we want each map $F_i$ to preserve the  complementary regions $\Omega_\pm$ in $\mathbb S^2 \setminus J$, i.e.~$F_i(U_i \cap \Omega_\pm) \subset \Omega_\pm$.

(iv) We require that the {\em Markov map} $F: J \to J$   defined by $F|_{J_i} = F_i$ is {\em mixing}, that is, for a sufficiently high iterate, we have $F^{\circ N}(J_i) = J$.

(v) Finally, we want the dynamics of $F$ to be {\em expanding}, i.e.~for some $N \ge 1$, we have $\inf_{z \in J} %\setminus E} 
|(F^{\circ N})'(z)| > 1$. (At the endpoints of the arcs  and their inverse orbits under $F$, we consider one-sided derivatives.)
%where $(F^{\circ N})'$ is not defined.)

This definition subsumes limit sets of quasi-Fuchsian groups and piecewise linear constructions such as snowflakes, see \cite{makarov99, mcmullen, PU, PUZ2}.
Note that for some purposes, one can allow $\bigcup J_i$ to be a proper subset of $J$; however, for connections to asymptotic variance, we must insist on the equality $J = \bigcup J_i$.

\end{remark}

\subsection{Dynamical families of conformal maps}

Given  $\mu \in M_{B}(\mathbb{D})$ with
$\|\mu\|_\infty \le 1$ as before, we may consider a natural holomorphic family of conformal maps
$H_t(z) = w^{t \mu}(z)$, $t \in \mathbb{D}$. We denote the associated dynamical systems and H\"older continuous potentials by $F_t$ and $\psi_t$ respectively.
In this formalism, $F_0 = B$.

If we restrict the parameter $t \in B(0, \rho)$ to a disk  of slightly smaller radius $\rho < 1$,
then H\"older bounds for quasiconformal mappings \cite[Theorem 3.10.2]{AIMb} imply the uniform estimate
\begin{equation}
\biggl\| \frac{\psi_t(z)}{t} \biggr\|_{C^\alpha(\mathbb{S}^1)} < \, K(\rho), \qquad \text{for some }0 < \alpha < 1.
\end{equation}
To prove Theorem \ref{thm:beta}, we consider the function
\begin{equation}
u(t) = \sigma^2 \biggl(\frac{\log H_t'}{t} \biggr), \qquad t \in \mathbb{D}.
\end{equation}
Observe that $u(t)$ extends continuously to the origin with $u(0) = \sigma^2(\mathcal S\mu)$. Indeed, the differentiability of the $\mathcal B$-valued analytic function $\log H_t'$ at the origin implies that
\begin{equation}
\biggl \| \frac{\log H_t'}{t} - \mathcal S\mu \biggr \|_{\mathcal B} =\mathcal O(|t|),
\end{equation}
from which the continuity of $u$ follows from the continuity of $\sigma^2(\cdot)$ in the Bloch norm.

\begin{theorem}
\label{thm:u-subharmonic}
The function $u(t)$
is real-analytic and subharmonic on the unit disk.
\end{theorem}
In particular, Theorem \ref{thm:u-subharmonic} shows that there exists a $t \in \mathbb{D}$, with $|t|$ arbitrarily close to 1, for which $u(t) \ge u(0)$.
Theorem \ref{dynamical-connections} implies
$$
 \liminf_{\tau \to 0} \frac{\beta_{H_t}(\tau)}{\tau^2/4} \ge |t|^2 \sigma^2(\mathcal S\mu).
$$
Taking the supremum over all eventually-invariant Beltrami coefficients $\mu$, $|t| \to 1$, 
 and using Theorem \ref{fractal-approximation} gives Theorem \ref{thm:beta}.

\begin{proof}[Proof of Theorem \ref{thm:u-subharmonic}]
We utilise the connection between $\sigma^2$ and the dynamical asymptotic variance. It is easy to see that the functions
$$u_n(t) = \var_n(\psi_t(z)/t), \qquad n = 1,2, \dots$$ are subharmonic. By the decay of correlations (\ref{eq:decay-corr}), the $u_n(t)$ converge uniformly to $u$
on compact subsets of the disk, hence $u(t)$ is subharmonic as well. The same argument can also be used to show the real-analyticity of $u$, for details, we refer the reader to \cite[Section 7]{PUZ2}.
\end{proof}

\subsection{Using higher-order terms}

 We now slightly refine the estimate from the previous section by taking advantage of the subharmonicity of the functions $\Delta^n u$ for $n \ge 1$. However, we do not know if these estimates improve upon Theorem \ref{thm:beta}, since the higher-order terms may be close to 0, when 
$\sigma^2(\mathcal S \mu)$ is close to $\Sigma^2$.

\begin{theorem}
One has
$$
\partial_t^j \overline{\partial}_t^k  u(t) = \sigma^2 \biggl( \partial_t^j \frac{  \log H_t'}{t}, \ \overline{\partial}_t^k \frac{  \log H_t'}{t} \biggr), \qquad t \in \mathbb{D}.
$$
\end{theorem}
\begin{proof}
We prove the statement by induction, one derivative at a time. For the first derivative,
$$
\partial_t  \var_n(\psi_t/t) = \frac{1}{n} \int_{\mathbb{S}^1} S_n (\partial_t (\psi_t/t)) \overline{S_n (\psi_t/t)} \, dm.
$$
Since  $t \to \psi_t/t$ is a bounded holomorphic map from $B(0,\rho)$ to the Banach space $C^\alpha(\mathbb{S}^1)$, the derivative $\partial_t(\psi_t/t)$ is H\"older continuous, in which case, the decay of correlations gives the convergence
$$
\partial_t  \var_n(\psi_t/t) \to \var( \partial_t (\psi_t/t), \, \psi_t/t), \quad \text{as } n\to\infty.
$$
To justify that  $\partial_t \var(\psi_t/t) =  \var( \partial_t (\psi_t/t), \, \psi_t/t)$, it suffices to use the well-known fact that if a sequence
of $C^1$ functions $F_n$ converges uniformly (on compact sets) to $F$, and the derivatives $F'_n$ converge uniformly to $G$, then necessarily $F'=G$.
One may compute further derivatives in the same way.
\end{proof} 

Leveraging the subharmonicity of the functions $\Delta^n u$ (which follows from the previous theorem by taking $j = k = n$), the Poisson-Jensen formula for subharmonic functions \cite[Theorem 4.5.1]{ransford} gives
\begin{equation}
\limsup_{r \to 1} \frac{1}{2\pi} \int_{|z|=r} u(z) \,|dz| \ \ge \  u(0) \, + \, \sum_{n=1}^\infty c_n^{-1} \cdot \Delta^n u(0),
\end{equation}
where $c_n = \Delta^n \bigl (|z|^{2n} \bigr)$.
As noted earlier, $u(0) = \sigma^2(\mathcal S\mu)$ while
$$
\frac{\Delta u(0)}{4} = \sigma^2 \Bigl (\mathcal S \mu \mathcal S \mu - \frac{1}{2} (\mathcal S \mu)^2 \Bigr )
$$
as the Neumann series expansion \eqref{neumann} shows.
The Beltrami coefficient $\mu$ from Lemma  \ref{mu} (with the choice of degree $d=16$) gives the value 
$$
\liminf_{\tau \to 0} \frac{B(\tau)}{\tau^2/4} \, \ge \,  \sigma^2(\mathcal S\mu) + \sigma^2 \Bigl (\mathcal S \mu \mathcal S \mu - \frac{1}{2} (\mathcal S \mu)^2 \Bigr ) \, >\, 0.893.
$$
Using further terms, and playing around with the parameters $(d, \rho_0, n_0)$, we were able to (rigorously) obtain the lower bound $0.93$ with the help of Mathematica to automate the computations.

\bibliographystyle{amsplain}

\begin{thebibliography}{00}

\bibitem{amo}
S.~Abenda, P.~Moussa, A.~H.~Osbaldestin,
{\it Multifractal dimensions and thermodynamical description of nearly-circular Julia sets},
Nonlinearity 12 (1999), no. 1, 19--40. 

\bibitem{stergios-theta}S.~Antonakoudis, \textit{Kra's theta conjecture}, preprint, 2014.

\bibitem{Arevista}
K.~Astala,
\textit{Calder\'on's problem for Lipschitz classes and the dimension of quasicircles}, Rev. Mat. Iberoamericana 4 (1988), no. 3-4, 469--486. 

\bibitem{A} K.~Astala,  \textit{Area distortion of quasiconformal mappings}, Acta Math. 173 (1994), no. 1, 37--60.

\bibitem{AIMb}
K.~Astala, T.~Iwaniec, G.~J.~Martin, \textit{Elliptic partial differential equations and quasiconformal mappings in the plane}, Princeton University Press, 2009.


\bibitem{AIPS} K.~Astala, T.~Iwaniec, I.~Prause, E.~Saksman,  \textit{Burkholder integrals, Morrey's problem and quasiconformal mappings}, Journal of Amer. Mat. Soc. 25 (2011), no. 2, 507--531.

\bibitem{AIPS2} K.~Astala, T.~Iwaniec, I.~Prause, E.~Saksman,  \textit{Bilipschitz and quasiconformal rotation, stretching and multifractal  spectra},  Publ. Math. Inst. Hautes \'Etudes Sci. (to appear).

%\bibitem{aps} K.~Astala, I.~Prause, S.~Smirnov, \textit{Harmonic measure and holomorphic motions}, in preparation.

\enlargethispage{4pt}

\bibitem{astalarohdeschramm95}
K.~Astala, S.~Rohde, O.~Schramm,  \textit{Self-similar Jordan curves}, in preparation.

\bibitem{BP}  J.~Becker, C.~Pommerenke,  \textit{On the Hausdorff dimension of quasicircles}, Ann. Acad. Sci. Fenn. Ser.
A I Math. 12 (1987), 329--333.

\bibitem{bersroyden} L.~Bers, H.~Royden, \textit{Holomorphic families of injections}, Acta Math. 157  (1986), 259--286.  

\bibitem{binder}I.~Binder, \textit{Harmonic measure and rotation spectra of planar domains}, preprint, 2008.

\bibitem{hedenmalm} H.~Hedenmalm,  \textit{Bloch functions and asymptotic variance}, preprint, 2015.

\bibitem{HK}H.~Hedenmalm, I.~R.~Kayumov,  \textit{On the Makarov law of integrated logarithm}, Proc. Amer. Math. Soc. 135, (2007), 2235--2248. 

\bibitem{IT}Y.~Imayoshi, M.~Taniguchi. \textit{An Introduction to Teichm\"uller Spaces}, Springer-Verlag, 1992.

\bibitem{ivrii-phd}O.~Ivrii, \textit{The geometry of the Weil-Petersson metric in complex dynamics}, preprint, 
\href{http://arxiv.org/abs/1503.02590}{arXiv:1503.02590}, 2015.

\bibitem{kalvuj} D.~Kalaj, D.~Vujadinovi\'c, \textit{Norm of the Bergman projection onto the Bloch space},  J. Operator Theory 73, no. 1, (2015), 113--126.

\bibitem{kayumov}
I.~R.~Kayumov, {\it The law of the iterated logarithm for locally univalent functions}, Ann. Acad. Sci. Fenn. Math. 27 (2002), no. 2, 357--364.

\bibitem{kuhnau}
R.~K\"uhnau, 
{\it M\"oglichst konforme Spiegelung an einer Jordankurve}, Jahresber. Deutsch. Math.-Verein. 90 (1988), no. 2, 90--109. 

\bibitem{le-zinsmeister}
T.~H.~N.~Le, M. Zinsmeister, \textit{On Minkowski dimension of Jordan curves}, Ann. Acad. Sci. Fenn. Math. 39 (2014), no. 2, 787--810.

\bibitem{lehto} O.~Lehto, \textit{Univalent functions and Teichm\"uller spaces}, Springer, 1987.

%\bibitem{makarov87}
%N.~G.~Makarov,
%{\it Conformal mapping and Hausdorff measures},  {Ark. Mat.} {25}  (1987),  no. 1, 41--89.

\bibitem{makarov90}
N.~G.~Makarov,
{\it Probability methods in the theory of conformal mappings}, (Russian) Algebra i Analiz 1 (1989), no. 1, 3--59; translation in Leningrad Math. J. 1 (1990), no. 1, 1--56. 

\bibitem{makarov99}
N.~G.~Makarov,
{\it Fine structure of harmonic measure}, 
{St. Petersburg Math. J.} {10} (1999), no. 2, 217--268.

\bibitem{MSS} 
R.~Ma{\~n}{\'e}, P. Sad, D. Sullivan, \textit{On the dynamics of rational maps}, 
Ann. Sci. \'Ecole Norm. Sup. (4) 16 (1983), no. 2, 193--217. 

\bibitem{mcm-amen}
C.~T.~McMullen, \textit{Amenability, Poincar\'e series and quasiconformal maps}, Invent. math. 97 (1989), 95--127.

\bibitem{mcm-cusps}
C.~T.~McMullen, 
{\it Cusps are dense},
Ann. Math. 133 (1991), 217--247. 

\bibitem{mcm-kahler}
C.~T.~McMullen, \textit{The moduli space of Riemann surfaces is K\"ahler hyperbolic}, Ann. Math. 151 (2000), 327--357.

\bibitem{mcmullen}
C.~T.~McMullen, 
{\it Thermodynamics, dimension and the Weil-Petersson metric},
Invent. Math. 173 (2008), no. 2, 365--425. 


\bibitem{PP}
W.~Parry, M.~Pollicott, {\it Zeta Functions and the Periodic Orbit Structure of Hyperbolic Dynamics}, Ast\'erisque, vol. 187--188, 1990.



\bibitem{perala}
A.~Per\"al\"a, {\it On the optimal constant for the Bergman projection onto the Bloch space}, Ann. Acad. Sci. Fenn. Math. 37 (2012), no. 1, 245--249.

\bibitem{pfaltz} J. A. Pfaltzgraff, \textit{Univalence of the integral of $f'(z)^\lambda$}, Bull. London
Math. Soc. 7  (1975),  254--256.

\bibitem{Pomm} C.~Pommerenke, \textit{Boundary behaviour of conformal maps},  
 Grundlehren der Mathematischen Wissenschaften 299, Springer-Verlag, 1992.

\bibitem{prause-smirnov} I.~Prause, S.~Smirnov, \textit{Quasisymmetric distortion spectrum}, Bull. Lond. Math. Soc. 43 (2011), 267--277.

\bibitem{ptut}
I.~Prause, X.~Tolsa, I.~Uriarte-Tuero, \textit{Hausdorff measure of quasicircles}, Adv. Math. 229 (2012), no. 2, 1313--1328.

\bibitem{PU} F.~Przytycki, M.~Urba\'nski, \textit{Conformal Fractals -- Ergodic Theory Methods}, 2009.

\bibitem{PUZ1}F.~Przytycki, M.~Urba\'nski, A.~Zdunik, \textit{Harmonic, Gibbs and Hausdorff measures on repellers for holomorphic maps I}, Ann. Math.
 130 (1989), 1--40.
\bibitem{PUZ2}F.~Przytycki, M.~Urba\'nski, A.~Zdunik, \textit{Harmonic, Gibbs and Hausdorff measures on repellers for holomorphic maps II}, Studia Math.
  97 (1991), 189--225.

\bibitem{ransford}T.~Ransford, \textit{Potential Theory in the Complex Plane}, London Mathematical Society Student Texts, 
Cambridge University Press, 1995.


\bibitem{reich} E.~Reich, \textit{Extremal quasiconformal mappings of the disk}, in 
{\em Handbook of Complex Analysis: Geometric Function Theory, vol. 1}, North-Holland, 75--136, 2002. 

\bibitem{ruelle}
D.~Ruelle, \textit{Repellers for real analytic maps}, 
Erg. Th. \& Dyn. Sys. 2 (1982), no. 1, 99--107. 

\bibitem{Slod} Z.~Slodkowski,  {\it Holomorphic motions and polynomial hulls}, Proc. Amer. Math. Soc. { 111}  (1991),
347--355.

\bibitem{smirnov}
S.~Smirnov, \textit{Dimension of quasicircles},
Acta Math. {205} (2010), no. 1, 189--197. 


\end{thebibliography}

\end{document}